\documentclass[10pt]{article}
\usepackage[english]{babel}

\usepackage[letterpaper,top=2cm,bottom=2cm,left=3cm,right=3cm,marginparwidth=1.75cm]{geometry}

\usepackage{amsmath,amsthm}
\usepackage{graphicx,xcolor}
\usepackage{verbatim,amsmath,amssymb}

\newcommand{\ZZ}{\mathbb{Z}}
\newcommand{\RR}{\mathbb{R}}
\newcommand{\st}{\textup{s.t.}}
\newcommand{\conv}{\textup{conv}}
\newcommand{\ap}{\textup{AP}}

\newcommand{\mtz}{\textup{MTZ}}
\newcommand{\dl}{\textup{DL}}

\newcommand{\scf}{\textup{SCF}}

\newcommand{\proj}{\textup{proj}}
\newcommand{\ones}{{\bf 1}}

\newcommand{\Cl}{\textup{{\bf Cl}}}
\newcommand{\Ef}{\textup{{\bf Ef}}}

\newcommand{\tq}{{\,:\,}}
\newcommand{\interior}{{\textup{int}}}

\newcommand{\PP}{{\mathcal{P}}}

\newcommand{\GA}[1]{{\color{blue} [GA] #1}}
\newcommand{\revcycle}[1]{{#1^\textup{R}}}

\newcounter{enucount}
\newtheorem{proposition}[enucount]{Proposition}
\newtheorem{theorem}[enucount]{Theorem}
\newtheorem{lemma}[enucount]{Lemma}
\newtheorem{corollary}[enucount]{Corollary}
\newtheorem{observation}[enucount]{Observation}
\theoremstyle{definition}
\newtheorem{definition}{Definition}

\title{On parametric formulations for the Asymmetric Traveling Salesman Problem}
\author{Gustavo Angulo\thanks{gangulo@uc.cl, Department of Industrial and Systems Engineering, Pontificia Universidad Católica de Chile, Santiago, Chile 7820436.}  \and
Diego A. Mor{\'{a}}n R.\thanks{morand@rpi.edu, Industrial and Systems Engineering Department,  Rensselaer Polytechnic Institute, Troy, NY USA 12180-3590.}   
} 

\begin{document}
\maketitle


\begin{abstract}
The traveling salesman problem is a widely studied classical combinatorial problem for which there are several integer linear formulations. In this work, we consider the Miller-Tucker-Zemlin (MTZ), Desrochers-Laporte (DL) and Single Commodity Flow (SCF) formulations. We argue that the choice of some parameters of these formulations is arbitrary and, therefore, there are families of formulations of which each of MTZ, DL, and SCF is a particular case. We analyze these families for different choices of the parameters, noting that in general the formulations involved are not comparable to each other and there is no one that dominates the rest. Then we define and study the closure of each family, that is, the set obtained by considering all the associated formulations simultaneously. In particular, we give an explicit integer linear formulation for the closure of each of the families we have defined and then show how they compare to each other.
\end{abstract}



\section{Introduction}
Let $G=(N,A)$ be a complete directed graph on $n\geq4$ nodes, where $N=\{1,\ldots,n\}$ and $A=\{ij\tq i,j\in N,\ i\neq j\}$. A Hamiltonian cycle or {\em tour} is a directed cycle in $G$ that begins and ends at the same node and such that each node is visited exactly once. The Asymmetric Traveling Salesman Problem (ATSP) seeks to find a Hamiltonian cycle of minimum cost with respect to a given vector $c\in\RR^A$. It finds a number or applications in logistics, sequencing, scheduling, among others. We refer the reader to \cite{applegate2006traveling} for a thorough coverage of history, applications, and solution approaches.

A generic mixed-integer linear programming formulation for the ATSP using binary variables $x_{ij}$ for $ij\in A$  can be written as
 \begin{align}
\min \quad&\sum_{ij\in A}c_{ij}x_{ij}\label{eq:obj_fun}\\
\text{s.t.}\quad &\sum_{ij\in\delta^+(i)} x_{ij} = 1\quad \textup{for all}\ i\in N\label{AssignPolytope1}\\ 
&\sum_{ji\in\delta^-(i)} x_{ji} = 1\quad \textup{for all}\ i\in N\label{AssignPolytope2}\\
&\{ij\in A\tq x_{ij}=1\}\  \text{does not contain subtours}\label{nosubtours}\\
&x_{ij} \in \{0,1\}\quad \textup{for all}\ ij\in A \label{AssignPolytope3},
\end{align}
where $\delta^+(i)$ is the set arcs originating from node $i$, $\delta^-(i)$ is the set of arcs arriving to node $i$ (see below for a formal definition of these sets of arcs) and a {\em subtour} is a cycle in $G$ that does not cover all nodes in $N$. In the optimization problem above, the objective function to be minimized is given by~\eqref{eq:obj_fun},  constraints \eqref{AssignPolytope1}, \eqref{AssignPolytope2} and \eqref{AssignPolytope3} define an integer linear programming formulation for a restricted version of the {\em Assignment Problem} (equivalent to the face of the assignment polytope given by the equations  $x_{ii}=0$ for all $i\in N$), while requirement \eqref{nosubtours} ensures that any feasible solution represents a single tour. 
Requirement \eqref{nosubtours} can be written using linear constraints on the $x$ variables (a formulation in the original space) or using additional variables and linear constraints (an extended formulation). 

In this paper, we study properties  of {\em parametric} integer programming formulations for the ATSP that are based on three classic formulations from the literature: the Miller-Tucker-Zemlin (MTZ), the Desrochers–Laporte (DL), and the Single-Commodity Flow (SCF) formulations. Before stating our results, we give some notation, precisely describe the MTZ, DL, and SCF formulations, and introduce the parametric formulations proposed in this paper: the $d$-MTZ, $d$-DL and $b$-SCF formulations. 

\paragraph{Notation:} For $S\subseteq N$, let $\delta^+(S)=\{ij\in A\tq i\in S,\ j\in N\setminus S\}$, $\delta^-(S)=\{ij\in A\tq i\in N\setminus S,\ j\in S\}$, and $A(S)=\{ij\in A\tq i\in S,\ j\in S\}$. If $S=\{i\}$, we write $\delta^+(i)$ and $\delta^-(i)$, respectively. Denote $N_1=N\setminus\{1\}=\{2,\ldots,n\}$ and $A_1=\{ij\tq i,j\in N_1,\ i\neq j\}$. 
Let $\mathcal C_1$ be the set of directed cycles with arcs in $A_1$ 
and let $\mathcal S_1$ be the set of subsets of $N_1$ of size at least 2. For a cycle $C$, its reversed cycle will be denoted $\revcycle{C}=\{ji\tq ij\in C\}$.

Given a set $Q=\{(x,u)\in\RR^{n_1}\times\RR^{n_2}\tq Ax+Bu\leq b\}$, its projection onto the $x$ variables is denoted as $\proj_x(Q)=\{x\in\RR^{n_1}\tq \exists u\in\RR^{n_2}\ \st\ (x,u)\in Q\}$.

The vector with all its components equal to one will be denoted $\ones$; its dimension must be understood by the context in which the notation is used. In addition, given a set of indices $K$, $\ones_K$ denotes a binary vector whose nonzero components are the ones indexed by $K$. Finally, $\RR_+$ and $\RR_{++}$ stand for the nonnegative and positive real numbers, respectively.

\subsection{Classic integer programming formulations for the ATSP}
Below we present some of the most commonly used integer programming formulations for the ATSP. We refer the reader to \cite{oncan2009comparative} for a thorough survey on different formulations for the ATSP and relations among them. In all formulations that follow, for modeling purposes and without loss of generality, we consider node 1 as a special node.

The Dantzig-Fulkerson-Johnson (DFJ) formulation \cite{dantzig1954solution} writes \eqref{nosubtours} in terms of the variables $x_{ij}$ for $ij\in A$ by using the exponentially many constraints known as {\em Clique inequalities}:
\begin{equation}\label{eq:DFJ_ineq}
\sum_{ij\in A(S)}x_{ij}\leq |S|-1\quad\textup{for all}\  S\in\mathcal S_1.
\end{equation}
The above inequalities can be replaced by the {\em Cut inequalities}:
\begin{equation*}\label{eq:DFJ_cut}
\sum_{ij\in\delta^+(S)}x_{ij}\geq 1\quad\textup{for all}\  S\in\mathcal S_1,
\end{equation*}
obtaining an equivalent formulation.

The Miller-Tucker-Zemlin (MTZ) formulation \cite{miller1960integer}, a compact extended formulation,  uses additional continuous variables $u_2,\ldots,u_n\in\RR$ and the quadratically many inequalities
\begin{equation}\label{eq:MTZ_ineq}
   u_i-u_j+(n-1)x_{ij}\leq n-2\quad\textup{for all}\ ij\in A_1.
\end{equation}
Variable $u_i$ for $i\in N_1$ can be understood as the relative position of node $i$ in the tour. Note that \eqref{eq:MTZ_ineq} for $ij\in A_1$ implies that if $x_{ij}=1$, then $u_i+1\leq u_j$.

The Desrochers–Laporte (DL) formulation \cite{desrochers1991improvements} strengthens the MTZ formulation by lifting variable $x_{ji}$ into \eqref{eq:MTZ_ineq}, yielding
\begin{equation}\label{eq:DL_ineq}
   u_i-u_j+(n-1)x_{ij}+(n-3)x_{ji}\leq n-2\quad\textup{for all}\ ij\in A_1.
\end{equation}
Note that \eqref{eq:DL_ineq} for $ij$ and $ji\in A_1$ imply that if $x_{ij}=1$, then $u_i+1= u_j$.

In \cite{gouveia1999asymmetric}, the $u$ variables in the MTZ and DL formulations are disaggregated in terms of binary variables that indicate precedence relations of nodes in the tour. The new variables $v_{ki}$ for $k,i\in N_1$ can be understood as indicating whether node $k$ precedes node $i$. The disaggregation of the MTZ formulation, termed RMTZ in \cite{gouveia1999asymmetric}, is given by
\begin{align}
x_{ij}+v_{ki}-v_{kj}\leq 1\quad &\textup{for all}\  ij\in A_1,\ k\in N_1\tq k\neq i,\ k\neq j\label{first_constraint_GP}\\
x_{ij}-v_{ij}\leq 0 \quad&\textup{for all}\  ij\in A_1\nonumber\\
x_{ij}+v_{ji}\leq 1 \quad &\textup{for all}\ ij\in A_1.\nonumber
\end{align}
Moreover, lifting constraint~\eqref{first_constraint_GP} to
\begin{equation}
    x_{ij}+x_{ji}+v_{ki}-v_{kj}\leq 1\nonumber
\end{equation}
yields a disaggregation of the DL formulation, termed L1RMTZ in \cite{gouveia1999asymmetric}.

The Single-Commodity Flow formulation (SCF) \cite{gavish1978travelling} includes additional flow variables $f_{ij}\in\RR_+$ along the constraints
\begin{eqnarray}
\sum_{ij\in\delta^+(i)}f_{ij}-\sum_{ji\in\delta^-(i)}f_{ji}=\begin{cases}n-1& i=1\\-1& i\in N_1\end{cases} \label{eq:scf_balance}\\
f_{ij}\leq (n-1)x_{ij}\quad \textup{for all}\ ij\in A. \label{eq:scf_bound}
\end{eqnarray}

Similarly to the MTZ and DL formulations, the SCF formulation can be disaggregated. The Multi-Commodity Flow formulation (MCF) \cite{wong1980integer} includes additional flow variables $f^k_{ij}\in\RR_+$ for each $k\in N_1$ and $ij\in A$ along the constraints
\begin{eqnarray}
\sum_{ij\in\delta^+(i)}f^k_{ij}-\sum_{ji\in\delta^-(i)}f^k_{ji}=\begin{cases}1& i=1\\-1&i=k\\0& i\in N_1,\ i\neq k\end{cases}\quad \text{\ for all}\ k\in N_1 \nonumber\\ 
f^k_{ij}\leq x_{ij}\quad \textup{for all}\ k\in N_1,\ ij\in A. \nonumber
\end{eqnarray}

\subsection{New parametric integer programming formulations for the ATSP}\label{section:our_formulations}

 We define new formulations for the ATSP  as follows: for each of the classic MTZ, DL, and SCF formulations, we keep the assignment polytope constraints unchanged and we replace some of the numbers appearing in the remaining constraints by appropriate parameters $d\in\RR^{A_1}_{++}$ or  $b\in\RR_{++}^{N_1}$, depending on the classic formulation being considered. 

The $d$-MTZ formulation is defined by replacing constraints~\eqref{eq:MTZ_ineq} in the MTZ formulation by
\begin{equation}\label{eq:genMTZ_ineq}
    u_i-u_j+d_{ij}\leq 1-x_{ij}\quad ij\in A_1.
\end{equation}

The $d$-DL formulation is given by the replacing the constraints~\eqref{eq:DL_ineq} in the DL formulation by
\begin{equation}\label{eq:genDL_ineq}
    u_i-u_j+x_{ij}+(1-d_{ij}-d_{ji})x_{ji}\leq 1-d_{ij}\quad \textup{for all}\ ij\in A_1.
\end{equation}

The $b$-SCF formulation is obtained by replacing  the constraints~\eqref{eq:scf_balance}-\eqref{eq:scf_bound} in the SCF formulation by
\begin{eqnarray}\label{eq:genSCF_ineq}
\sum_{ij\in\delta^+(i)}f_{ij}-\sum_{ji\in\delta^-(i)}f_{ji}=\begin{cases}1& i=1\\-b_i& i\in N_1\end{cases}\label{eq:genSCF1}\\
f_{ij}\leq  x_{ij}\quad \textup{for all}\ ij\in A\label{eq:genSCF2}.
\end{eqnarray}

We remark here that in order for these to be valid formulations for the ATSP, the parameters $d\in\RR^{A_1}_{++}$ and $b\in\RR_{++}^{N_1}$ must be chosen appropriately. We define
$$D=\left\{d\in\RR^{A_1}_{++}\tq \sum_{ij\in C}d_{ij}\leq 1\ \forall C\in\mathcal C_1\right\}$$
and
$$B=\left\{b\in\RR^{N_1}_{++}\tq \sum_{i\in N_1}b_i=1\right\}.$$
In the sections corresponding to each of the proposed formulations, we argue that \eqref{eq:genMTZ_ineq} or \eqref{eq:genDL_ineq} with $d\in D$, and \eqref{eq:genSCF_ineq} with $b\in B$ provide valid formulations for the ATSP.

\paragraph{Relevant polyhedral sets:}
Each formulation that we have presented has a polyhedron associated to it which is obtained by removing all integrality constraints from the corresponding formulation. 
We denote the polyhedron associated to the (restricted) Assignment Problem as $P_\ap=\{x\in[0,1]^A\tq  x\ \text{satisfies \eqref{AssignPolytope1} and \eqref{AssignPolytope2}}\}$.
Similarly, we define the polyhedra associated to the parametric formulations: for the $d$-MTZ formulation we let $Q_{\mtz}(d)=\{(x,u)\in P_\ap\times \RR^{N_1}\tq  (x,u)\ \text{satisfies \eqref{eq:genMTZ_ineq}}\}$, for the DL formulation we let $Q_{\dl}(d)=\{(x,u)\in P_\ap\times \RR^{N_1}\tq  (x,u)\ \text{satisfies \eqref{eq:genDL_ineq}}\}$, and for the SCF formulation we let $Q_{\scf}(b)=\{(x,f)\in P_\ap\times \RR^{A}_+\tq  (x,f)\ \text{satisfies \eqref{eq:genSCF1} and \eqref{eq:genSCF2}}\}$. We also let $P_\mtz(d)=\proj_x(Q_\mtz(d))$, $P_\dl(d)=\proj_x(Q_\dl(d))$ and $P_\scf(b)=\proj_x(Q_\scf(b))$.
\subsection{Our results and organization of the paper}\label{section:our_results}
In this paper, we study the following properties of the formulations defined in the previous section: 

\begin{enumerate}
    \item {\bf Characterization of the projection onto the $x$-variables space.} We study the formulations $P_\mtz(d)$, $P_\dl(d)$, and $P_\scf(b)$ that are obtained by projecting the extended formulations $Q_\mtz(d)$, $Q_\dl(d)$, and $Q_\scf(b)$ onto the space of $x$-variables. In particular, we give a full polyhedral description and characterize their facets.

    \item {\bf Comparing the formulations for different parameters.} Two formulations are comparable if one it is included in the other (the formulation that is included in the other is said to be stronger). We show that in general for $d,d'\in D$, $P_\mtz(d)$ and $P_\mtz(d')$ are not comparable; we also give conditions for which given $d$ we can find a parameter $d'$ that gives a stronger formulation. We also show that for $d,d'\in D$,  under a minor condition, the formulations $P_\dl(d)$ and $P_\dl(d')$ are not comparable. Finally, we show that for $b,b'\in B$, the formulations $P_\scf(d)$ and $P_\scf(d')$ are never comparable.
    \item {\bf Characterizing the closures.} We define the closure of a family of parametric formulations as the set obtained by simultaneously considering the formulations for all values of the parameters. More precisely, for the $d$-$\mtz$ formulations the closure is the set $\bigcap_{d\in D} P_\mtz(d)$, for the $d$-$\dl$ formulations the closure is the set $\bigcap_{d\in D} P_\dl(d)$ and for the $b$-$\scf$ formulations the closure is the set $\bigcap_{b\in B} P_\scf(d)$. We completely characterize all these closures and study some of their properties.
\end{enumerate}

The rest of the paper is organized as follows: in Section~\ref{sec:prelim} we define general parametric formulations, the closure of a family of  parametric formulations and related concepts, and study their properties. In this section we also study properties of the sets of parameters $D$ and $B$, and properties of classic formulations. In Section~\ref{sec:mtz}, Section~\ref{sec:dl} and Section~\ref{sec:scf} we study the $d$-MTZ formulations, $d$-DL formulations, and $b$-SCF formulations, respectively. Finally, in Section~\ref{sec:comp_closures} we compare the closures of the parametric formulations studied in this paper.

 \section{Preliminaries}\label{sec:prelim}

\subsection{General parametric formulations and extended formulations}
 Since in this paper we consider several formulations for the ASTP, it is convenient to study  properties of formulations and extended formulations of sets of integral vectors in more generality. 
\begin{definition}[Formulations]
  A formulation for a set $S\subseteq\ZZ^{n_1}$ is a set $P\subseteq\RR^{n_1}$ such that $S=P\cap\ZZ^{n_1}$. An extended formulation of $S$ is a set $Q\subseteq\RR^{n_1}\times E$, where $E$ is an arbitrary set, such that $S=\proj_{x}(Q)\cap\ZZ^{n_1}$.   
\end{definition}


The parametric (extended) formulations we study are polyhedra, and the parameters can appear in the constraint matrix or in the r.h.s.. The following definition encompasses all the formulations we consider and gives a precise way to compare formulations for different parameters. 

\begin{definition}[Parametric Formulations, Domination]
Let $S\subseteq\ZZ^{n_1}$, $X\subseteq\RR^{n_1}$,  and let $\PP\subseteq (\RR^{m\times n_1}\times \RR^{m\times n_2} \times\RR^m)$ be a set of parameters. Given $(A,G,b)\in \PP$, we consider the set $P(A,G,b)=\{x \in X\tq \exists u\in\RR^{n_2}\ \textup{s.t.}\ A x+Gu\leq b\}$. We say that $P(A,G,b)$ with the set $\PP$ is a parametric formulation for $S$ if $P(A,G,b)\cap\ZZ^{n_1}= S$ for all $(A,G,b)\in \PP$.

For $(A,G,b),(A',G',b')\in \PP$ we say that $P(A,G,b)$ dominates $P(A',G',b')$ if $P(A,G,b)\subseteq P(A',G',b')$; we say that $P(A,G,b)$ and $P(A',G',b')$ are comparable if one of these formulations dominates the other.
\end{definition}

In the context of our paper, we have $X=P_\ap$ (the restricted assignment polytope). For the $d$-$\mtz$ and $d$-$\dl$ formulations (respectively $b$-$\scf$ formulations) we have that the associated set of parameters is $\PP=D$ (respectively $\PP=B$), and the parametric formulations are given by the polyhedra $P_\mtz(d)$ and $P_\dl(d)$ that only depend on the parameter $d\in D$ (respectively the polyhedron $P_\scf(b)$ that only depends on the parameter $b\in B$). 

\begin{observation}\label{obs:Q_different_u}
Let $S\subseteq\ZZ^{n_1}$, $X\subseteq\RR^{n_1}$, and $\PP\subseteq (\RR^{m\times n_1}\times \RR^{m\times n_2} \times\RR^m)$, and consider the parametric formulation given by  $P(A,G,b)$ and $(A,G,b)\in \PP$. Denote $Q(A,G,b)=\{(x,u) \in X\times\RR^{n_2}\tq  A x+Gu\leq b\}$ for $(A,G,b)\in \PP$. Then we can write $P(A,G,b)=\{x \in X\tq \exists u\in\RR^{n_2}\ \textup{s.t.}\ (x,u)\in Q(A,G,b)\}$. Observe that for $(A,G,b), (A',G',b')\in \PP$ we have that $P:=P(A,G,b)\cap P(A',G',b')$ is a formulation for $S$, but in general $Q:=Q(A,G,b)\cap Q(A',G',b')$ does not define a formulation for $S$, as we have $\proj_x(Q)\cap \ZZ^{n_1} \subseteq P\cap \ZZ^{n_1}=S$ and  the inclusion may be strict. However, the set $\hat Q=\{(x,u,u')\in X\times\RR^{n_2}\times \RR^{n_2}\tq  A x+Gu\leq b,\ A'x+G'u'\leq b'\}$ satisfies $\proj_x(\hat Q)=P$ and thus  defines a formulation for $S$ (we show this property for arbitrary intersections in Lemma~\ref{lemma:operators_relation} below).
\end{observation}

In addition to understanding when a formulation $P(A,G,b)$ dominates a formulation $P(A',G',b')$ for different parameters $(A,G,b)$ and $(A',G',b')$, we are also interested in understanding the set obtained by considering the formulations $P(A,G,b)$ for all parameters $(A,G,b)$ simultaneously. More precisely, we are interested in the following set. 

\begin{definition}[Closure of a family of parametric formulations]
Let $X\subseteq\RR^{n_1}$, $\PP\subseteq (\RR^{m\times n_1}\times \RR^{m\times n_2} \times\RR^m)$, and $P(A,G,b)=\{x \in X\tq \exists u\in\RR^{n_2}\ \textup{s.t.}\ A x+Gu\leq b\}$ for $(A,G,b)\in \PP$.  The closure of a family of parametric formulations w.r.t. the set of parameters $\PP$, denoted $\Cl(P(\PP))$, is the set obtained by intersecting all the formulations $P(A,G,b)$ with $(A,G,b)\in \PP$, that is,
$$\Cl(P(\PP))=\bigcap_{(A,G,b)\in \PP}P(A,G,b)=\{x \in X\tq \textup{for all}\ (A,G,b)\in \PP,\ \exists u\in\RR^{n_2}\ \textup{s.t.}\ A x+Gu\leq b\  \}.$$
\end{definition}

The following classic result from the Robust Optimization literature (see~\cite{ROBook2009}) allows us to restrict our analysis to sets of parameters $\PP$ that are closed and convex, and also gives a sufficient condition for the closure $\Cl(P(\PP))$ to be a polyhedral set. 

\begin{lemma}[see~\cite{ROBook2009}]\label{lem:robustprops}
Let $X\subseteq\RR^n_1$, $\PP\subseteq (\RR^{m\times n_1}\times \RR^{m\times n_2} \times\RR^m)$, and $P(A,G,b)=\{x \in X\tq \exists u\in\RR^{n_2}\ \textup{s.t.}\ A x+Gu\leq b\}$ for all $(A,G,b)\in \PP$.
\begin{enumerate}
    \item The sets of parameters $\PP$ and the closure of the convex hull of $\PP$ give the same intersection:
$$\Cl(P(\PP))=\Cl(P(\overline{\conv(\PP)})).$$
\item If $X$ and $\PP$ are polyhedra, then $
\Cl(P(\PP))$ is a polyhedron.
\end{enumerate}
\end{lemma}

For the $d$-$\mtz$, $d$-$\dl$, and $b$-$\scf$ parametric formulations, the set of parameters $\PP$ is not a polyhedron, as neither $D$ or $B$ are closed sets. However, since their closure is a polytope, by part (1.) of Lemma \ref{lem:robustprops} we will be able to work with the polytopes $\overline{D}$ and $\overline{B}$ when computing the closures of the associated parametric formulations.

The closure of a family of parametric formulations is defined by intersecting all the formulations $P(A,G,b)$ in the $x$-space. We are also interested in studying sets in the $x$-space that are obtained from extended formulations that consider an arbitrary number of variables. This is motivated by the study of the ``intersection'' of all the formulations $Q(A,G,b)$, which, by Observation~\ref{obs:Q_different_u}, must be done by considering a copy of the variable $u$ for each parameter $(A,G,b)$ in order to obtain a valid formulation in the $x$-space.

\begin{definition}[Extended formulation with arbitrary number of variables] Let $X\subseteq\RR^{n_1}$ and $\PP\subseteq (\RR^{m\times n_1}\times \RR^{m\times n_2} \times\RR^m)$ be a set of parameters. Denote $Q(A,G,b)=\{(x,u) \in X\times\RR^{n_2}\tq  A x+Gu\leq b\}$ for $(A,G,b)\in \PP$. We define
$$\Ef(Q(\PP))=\{(x,u)\in X\times\left(\RR^{n_2}\right)^{\PP}\tq  A x+Gu(A,G,b)\leq b\  \textup{for all}\ (A,G,b)\in \PP\}.$$
Here, we identify $\left(\RR^{n_2}\right)^{\PP}$ with the set of functions $u\tq \PP\rightarrow \RR^{n_2}$, which in general is an infinite-dimensional vector space. 
\end{definition}

If $\PP$ is finite in the above definition, we recover $\RR^{n_2\times \PP}$, whose dimension is $n_2|\PP|$. In particular, if $\PP=\{(A^1,G^1,b^1),\ldots,(A^k,G^k,b^k)\}$, we can write
$$\Ef(Q(\PP))=\{(x,(u^1,\ldots,u^k))\in X\times\left(\RR^{n_2}\right)^k\tq  A^i x+G^iu^i\leq b^i\  \textup{for all}\ i=1,\ldots,k\}.$$

 Let $(A,G,b)\in \PP$. Note that by definition we have $P(A,G,b)=\proj_x(Q(A,G,b))$. A similar relation can be established between the operators $\Cl(\cdot)$ and $\Ef(\cdot)$ when considering all the formulations with parameters in the set $\PP$.

\begin{lemma}\label{lemma:operators_relation} Let $X\subseteq\RR^n_1$, $\PP\subseteq (\RR^{m\times n_1}\times \RR^{m\times n_2} \times\RR^m)$, and for $(A,G,b)\in \PP$, let $P(A,G,b)=\{x \in X\tq \exists u\in\RR^{n_2}\ \textup{s.t.}\ A x+Gu\leq b\}$ and  $Q(A,G,b)=\{(x,u) \in X\times\RR^{n_2}\tq  A x+Gu\leq b\}$  .
    We have that
    $$\proj_x(\Ef(Q(\PP)))=\Cl(P(\PP))$$
\end{lemma}

\begin{proof}
    We have
    \begin{align*}
        \proj_x(\Ef(Q(\PP)))&=\{x\in X\tq  \exists\ u\in \left(\RR^{n_2}\right)^{\PP}\ \textup{s.t.}\   A x+Gu(A,G,b)\leq b\  \textup{for all}\ (A,G,b)\in \PP\}\\
        &=\{x \in X\tq \textup{for all} \ (A,G,b)\in \PP,\ \exists\ u\in\RR^{n_2}\ \textup{s.t.}\ A x+Gu\leq b\}\\
        &=\bigcap_{(A,G,b)\in \PP}\{x \in X\tq \exists u\in\RR^{n_2}\ \textup{s.t.}\ A x+Gu\leq b\}\\
        &=\bigcap_{(A,G,b)\in \PP}P(A,G,b)=\Cl(P(\PP)).
    \end{align*}
\end{proof}

As a corollary of the previous results, we obtain that if the set of parameters $\PP$ is a polytope, then the associated closure and the extended formulation with arbitrary number of variables have a finite representation in terms of the vertices of the polytope.

\begin{corollary}\label{prod_V}
    If $\mathcal P$ is a nonempty polytope with vertices $\mathcal V$, then $\Cl(P(\mathcal P))=\Cl(P(\mathcal V))$ and $\Ef (Q(\mathcal P)) = \Ef (Q(\mathcal V))$.
\end{corollary}
\begin{proof}
    Since $\mathcal P=\conv(\mathcal V)$, by (1.) of Lemma \ref{lem:robustprops}, we obtain $\Cl(P(\mathcal P))=\Cl(P(\mathcal V))$. Now, by Lemma \ref{lemma:operators_relation}, we have $\proj_x(\Ef(Q(\PP)))=\Cl(P(\PP))$ for any set of parameters $\PP$, and therefore $$\proj_x(\Ef(Q(\mathcal P)))=\Cl(P(\mathcal P))=\Cl(P(\mathcal V))=\proj_x(\Ef(Q(\mathcal V))).$$
\end{proof}

\begin{observation}
    If $P(A,G,b)$ is a formulation and $Q(A,G,b)$ is an extended formulation of a set $S$ for all $(A,G,b)\in \PP$, it is easy to see that $\Cl(P(\PP))$ is a formulation of $S$ and that $\Ef(Q(\PP))$ is an extended formulation of $S$ with $E=\left(\RR^{n_2}\right)^{\PP}$. As in Observation~\ref{obs:Q_different_u}, we have that the set $\{(x,u)\in X\times\RR^{n_2}\tq  A x+Gu\leq b\  \textup{for all}\ (A,G,b)\in \PP\}$, where we use the common variable $u\in\RR^{n_2}$ over all parameters in $\PP$, in general is not a formulation of $S$. Finally, $\Cl(P(\PP))$ (resp. $\Ef(Q(\PP))$ can be a formulation (resp. extended formulation) of $S$ even if $P(A,G,b)$ (resp. $Q(A,G,b)$)  is not a formulation (resp. extended formulation) for some $(A,G,b)\in \PP$ (see Proposition~\ref{prod_mtz} in Section \ref{sec:closure_d-MTZformulations}, Proposition~\ref{prod_dl} in Section \ref{sec:closure_d-DLformulations}, and Proposition~\ref{prod_SCF} in Section \ref{sec:closure_d-SCFformulations}).
\end{observation}

\subsection{Properties of the sets of parameters $D$ and $B$}

Let $\overline D$ be the topological closure of $D$, that is, $\overline D=\{d\in\RR^{A_1}_+\tq \sum_{ij\in C}d_{ij}\leq 1\ \textup{for all}\ C\in\mathcal C_1\}$.

\begin{proposition}
$\overline{D}$ is a full-dimensional polytope and all its inequalities are non-redundant and thus facet-defining.
\end{proposition}
\begin{proof}
Note that as $\overline{D}\subseteq\RR^{A_1}_{+}$, for any $d\in \overline{D}$ we have $d_{ij}\geq0$ for all  $ij\in A_1$. On the other hand, since $G$ is a complete directed graph, every arc $ij\in A_1$ belongs to some cycle in $\mathcal C_1$. Hence,  variable $d_{ij}$ appears in at least one inequality defining $\overline{D}$, and thus $d_{ij}\leq 1$ for any $d\in \overline{D}$. This shows that $\overline{D}\subseteq [0,1]^{A_{1}}$, and therefore $\overline{D}$ is bounded, hence, a polytope. On the other hand, since the point $\frac{1}{|A_1|+1}\ones\in\RR^{A_1}_{+}$  satisfies all inequalities defining $\overline{D}$ strictly, we conclude $\overline{D}$ is a full-dimensional set. 
    
Next we show that for each $C\in \mathcal{C}_1$ the inequality $\sum_{ij\in C}d_{ij}\leq 1$ is facet-defining. As $\overline{D}$ is full-dimensional, it suffices to show that this inequality is nonredundant. To see this, observe that the point $\frac{1}{|C|-1}\ones_C\in\RR^{A_1}_{+} $ satisfies all inequalities defining  $\overline{D}$ except for the one associated to $C$.
\end{proof}

 \begin{proposition}\label{nphard}
 The separation problem with respect to $\overline D$ is NP-hard.
\end{proposition}

\begin{proof}
  Let $\mathcal H_1\subseteq \mathcal C_1$ be the set of Hamiltonian cycles on $N_1$. We consider an instance of longest Hamiltonian cycle given by $d\in\ZZ^{A_1}_+$ and a lower bound $k\in\ZZ_+$. The instance has answer Yes if and only if there exists $C\in\mathcal H_1$ such that $\sum_{ij\in C}d_{ij}\geq k+1$, where this condition is equivalent to $\sum_{ij\in C}d_{ij}>k$. Without loss of generality, we can assume that $d\leq k\ones$, for otherwise the answer is trivially Yes.

   Let $d'=d+k\ones$. We have $\max\{\sum_{ij\in C} d'_{ij}\tq C\in\mathcal H_1\}=\max\{\sum_{ij\in C} d_{ij}\tq C\in\mathcal H_1\}+k|N_1|$. Note that $d_{ij}'=d_{ij}+k\leq 2k\leq 2k+d_{il}+d_{lj}=d_{il}'+d_{lj}'$ for distinct $i,j,l\in N_1$. Therefore, $d'$ satisfies the triangle inequality and hence $\max\{\sum_{ij\in C} d'_{ij}\tq C\in\mathcal C_1\}=\max\{\sum_{ij\in C}d'_{ij}\tq C\in\mathcal H_1\}$, which implies $\max\{\sum_{ij\in C} d'_{ij}\tq C\in\mathcal C_1\}=\max\{\sum_{ij\in C} d_{ij}\tq C\in\mathcal H_1\}+k|N_1|$.
   In particular, $\max\{\sum_{ij\in C} d_{ij}\tq C\in\mathcal H_1\}>k$ if and only if $\max\{\sum_{ij\in C} d'_{ij}\tq C\in\mathcal C_1\}>k(|N_1|+1)$.
   Let $\hat d=\frac{1}{k(|N_1|+1)}d'$. We have $\hat d\notin \overline D$ if and only if there exists $C\in\mathcal C_1$ such that $\sum_{ij\in C}\hat d_{ij}>1$, which is equivalent to $\max\{\sum_{ij\in C} d_{ij}\tq C\in\mathcal H_1\}>k$. Since building $\hat d$ takes linear time, we have a polynomial reduction.
\end{proof}

\begin{observation}\label{obs_nphard}
    Proposition~\ref{nphard} implies that enumerating the vertices of $\overline D$ is unlikely to be a tractable problem. On the other hand, letting $\overline B=\{b\in \RR^{N_1}_+\tq \sum_{i\in N_1} b_i=1\}$, we have that $\overline B$ is an $(n-2)$-dimensional polytope whose vertices are the $n-1$ canonical vectors in $\RR^{N_1}$.
\end{observation}

\subsection{Properties of classic formulations for the ATSP}


Since the extended formulations presented in Section \ref{section:our_formulations} are given by different sets of additional variables (in addition to the  binary variables $x_{ij}$ for  $ij\in A$), we consider their projection onto the $x$-space to compare them. The following lemma in the spirit of \cite{velednitsky2017short} will prove useful.

\begin{lemma}\label{lem:projrequiem}
Let $X\subseteq\RR^A$. For each $ij\in A_1$, let $\alpha^{ij}\in\RR^A$ and $\beta^{ij}\in\RR$, and consider the set
$$Q=\left\{(x,u)\in X\times\RR^{N_1}\tq  u_i-u_j+\sum_{kl \in A}\alpha^{ij}_{kl}x_{kl}\leq \beta^{ij}\ \textup{for all}\ ij\in A_1\right\}.$$
Then $\proj_{x}(Q)=P$, where
$$P=\left\{x\in X\tq \sum_{ij\in C}\sum_{kl\in A}\alpha^{ij}_{kl}x_{kl}\leq \sum_{ij\in C}\beta^{ij}\ \textup{for all}\  C\in\mathcal{C}_1\right\}.$$
\end{lemma}
\begin{proof}
If $x\in\proj_x(Q)$, then there exists $u\in\RR^{N_1}$ such that $(x,u)\in Q$. Therefore, for any $C\in\mathcal C_1$, $(x,u)$ satisfies $u_i-u_j+\sum_{kl \in A}\alpha^{ij}_{kl}x_{kl}\leq \beta^{ij}$ for all $ij\in C$, which implies $\sum_{ij\in C}\sum_{kl\in A}\alpha^{ij}_{kl}x_{kl}\leq \sum_{ij\in C}\beta^{ij}$. Hence, $x\in P$.

Now, given $x\in P$, let $c_{ij}=\beta^{ij}-\sum_{kl \in A}\alpha^{ij}_{kl}x_{kl}$ for $ij\in A_1$. We have $\sum_{ij\in C}c_{ij}\geq 0$ for all $C\in\mathcal C_1$, and thus there are no cycles in $\mathcal C_1$ with negative length with respect to $c$. Fix $h\in N_1$ and for each $i\in N_1$, let $-u_i\in\RR$ be the length of a shortest path from $h$ to $i$ with respect to $c$. Since there are no negative-length cycles, we have $-u_j\leq -u_i+c_{ij}=-u_i+\beta^{ij}-\sum_{kl \in A}\alpha^{ij}_{kl}x_{kl}$ for all $ij\in A_1$. Therefore, $u\in\RR^{N_1}$ is such that $(x,u)\in Q$, and hence $x\in\proj_x(Q)$.
\end{proof}

Projecting onto the $x$-variables the different extended formulations, we obtain formulations on the $x$-variables that have different constraints whose comparison allows to analyze the strength of the ATSP formulations we consider. 

We first observe that the {\em Clique inequalities} \eqref{eq:DFJ_ineq} in the DFJ formulation dominate the  {\em Circuit inequalities}
\begin{equation}\label{eq:circuit_ineq}
\sum_{ij\in C}x_{ij}\leq |C|-1\quad\textup{for all}\   C\in\mathcal C_1
\end{equation}
and the {\em Weak clique inequalities}
\begin{equation}\label{eq:weak_clique_ineq}
\sum_{ij\in A(S)}x_{ij}\leq |S|-\frac{|S|}{n-1}\quad\textup{for all}\  S\in\mathcal S_1.
\end{equation}
In turn, the {\em Circuit inequalities} dominate the {\em Weak circuit inequalities}
$$\sum_{ij\in C}x_{ij}\leq |C|-\frac{|C|}{n-1}\quad\textup{for all}\  C\in\mathcal C_1.$$

By using Lemma \ref{lem:projrequiem}, it can be shown that the projection of the MTZ and RMTZ formulations onto the $x$-variables is given by the {\em Weak circuit inequalities} \cite{padberg1991analytical} and  {\em Circuit inequalities} \cite{gouveia1999asymmetric}, respectively, while projecting the SCF and MCF formulations yields the {\em Weak clique inequalities} \cite{padberg1991analytical} and {\em Clique inequalities} \cite{wong1980integer}, respectively. In particular, MCF is equivalent to DFJ. The projection of the DL formulation gives rise to a class of {\em Lifted weak circuit inequalities} \cite{bektacs2014requiem}
\begin{equation}\sum_{ij\in C}x_{ij}+\frac{n-3}{n-1}\sum_{ij\in \revcycle{C}} x_{ij}\leq |C|-\frac{|C|}{n-1}\quad\textup{for all}\  C\in\mathcal C_1,\ |C|\geq 3. \nonumber\end{equation}

\begin{observation}\label{obs_c1}
Note that \eqref{AssignPolytope1} and \eqref{AssignPolytope2} imply $\sum_{ij\in A}x_{ij}=n$ and $\sum_{ij\in\delta^+(1)\cup\delta^-(1)}x_{ij}=2$, and thus $\sum_{ij\in A_1}x_{ij}=n-2$. Therefore, the \emph{Clique inequalities} \eqref{eq:DFJ_ineq} for $S\in\mathcal S_1$ with $|S|=n-2$ are redundant, which in turn implies that the above inequalities are redundant as well for $C\in\mathcal C_1$ with $|C|=n-2$.

\end{observation}





\section{Parametric $d$-MTZ formulations}\label{sec:mtz}


 Given $d\in\RR^{A_1}_{++}$ and a constant $M>0$, we consider the following generalization of the MTZ inequalities \eqref{eq:MTZ_ineq}:
\begin{equation}\label{eq:genMTZ_ineq_M}
    u_i-u_j+d_{ij}\leq M(1-x_{ij})\quad ij\in A_1.
\end{equation}
For large enough $M$, inequalities \eqref{eq:genMTZ_ineq_M} define a valid extended formulation for the ATSP that we call $d$-MTZ formulation: for $x_{ij}=0$, \eqref{eq:genMTZ_ineq_M} reduces to the superfluous constraint $u_i-u_j+d_{ij}\leq M$, whereas for $x_{ij}=1$, it implies $u_j\geq u_i + d_{ij}>u_i$ and thus prohibits $x$ from defining cycles in $\mathcal C_1$. Constant $M$ can be chosen appropriately as $M=M(d)\geq \max\{\sum_{ij\in C}d_{ij}\tq C\in\mathcal C_1\}$. We recover the inequalities \eqref{eq:MTZ_ineq} defining the original MTZ formulation when we set $d_{ij}=1$ for all $ij\in A_1$ and $M=n-1$.\\

We want to study the effect of different choices of $d\in\RR^{A_1}_{++}$ on the $d$-MTZ formulation. For this purpose, it is convenient to normalize $d$ by choosing $d\in D=\{d\in\RR^{A_1}_{++}\tq \sum_{ij\in C}d_{ij}\leq 1\ \textup{for all}\ C\in\mathcal C_1\}$ so that $M=1$ defines a valid ATSP formulation. We obtain the (normalized) generalized MTZ-like inequalities~\eqref{eq:genMTZ_ineq}:
 \begin{equation*}
     u_i-u_j+d_{ij}\leq 1-x_{ij}\quad ij\in A_1.
 \end{equation*}
For instance, the vector $d^{\mtz}\in D$ defined as $d^{\mtz}_{ij}=\frac{1}{n-1}$ for $ij\in A_1$ gives the classic MTZ inequalities in normalized form:
 $$u_i-u_j+\frac{1}{n-1}\leq 1-x_{ij}.$$

For $d\in D$, recall the set $Q_\mtz(d)$, the continuous relaxation of the extended ATSP formulation obtained when using inequalities \eqref{eq:genMTZ_ineq}:
$$Q_\mtz(d)=\left\{(x,u)\in P_\ap\times\RR^{N_1}\tq u_i-u_j+d_{ij}\leq 1-x_{ij}\ \text{for all }ij\in A_1\right\}.$$
While $d\in \overline D\setminus D$ does not define a valid formulation, for convenience we extend the above definition of $Q_\mtz(d)$ to $d\in\overline D$ as well.

\subsection{Obtaining the $d$-MTZ formulation from convexification}

In this section, we argue that the proposed generalization is a natural and sensible choice. To that end, below we generalize and strengthen Proposition 1 in \cite{bektacs2014requiem}.

For $d\in D$ and $ij\in A_1$, let $Z_\mtz^{ij}(d)$ be the set of vectors $(u_i,u_j,x_{ij},x_{ji})\in\RR^2\times\{0,1\}^2$ such that $x_{ij}+x_{ji}\leq 1$ and
\begin{itemize}
    \item $[x_{ij}=1,\ x_{ji}=0]\Rightarrow u_i+d_{ij}\leq u_j,\ u_j-u_i\leq 1-d_{ji}$
    \item $[x_{ij}=0,\ x_{ji}=1]\Rightarrow u_j+d_{ji}\leq u_i,\ u_i-u_j\leq 1-d_{ij}$
    \item $[x_{ij}=0,\ x_{ji}=0]\Rightarrow u_i-u_j\leq 1-d_{ij},\ u_j-u_i\leq 1-d_{ji}$.
\end{itemize}

\begin{proposition}
The convex hull of $Z_\mtz^{ij}(d)$ is given by
\begin{eqnarray*}
u_i-u_j+x_{ij}\leq 1-d_{ij}\\
u_j-u_i+x_{ji}\leq 1-d_{ji}\\
x_{ij}+x_{ji}\leq 1\\
x_{ij}\geq 0\\
x_{ji}\geq 0.
\end{eqnarray*}
\end{proposition}
\begin{proof}
Let us write $Z_\mtz^{ij}(d)$ as the union of the following sets:
\begin{itemize}
    \item $\{(u_i,u_j)\in\RR^2\tq u_i+d_{ij}\leq u_j,\ u_j-u_i\leq 1-d_{ji}\}\times\{(1,0)\}$
    \item $\{(u_i,u_j)\in\RR^2\tq u_j+d_{ji}\leq u_i,\ u_i-u_j\leq 1-d_{ij}\}\times\{(0,1)\}$
    \item $\{(u_i,u_j)\in\RR^2\tq u_i-u_j\leq 1-d_{ij},\ u_j-u_i\leq 1-d_{ji}\}\times\{(0,0)\}$
\end{itemize}
By disjunctive programming \cite{balas1979disjunctive}, the convex hull of $Z_\mtz^{ij}(d)$ is given by the projection on the $(u_i,u_j,x_{ij},x_{ji})$ variables of the system
\begin{align*}
u_i=v_i^1+v_i^2+v_i^3& &v_i^1-v_j^1+d_{ij}\lambda_1&\leq 0 \\
   u_j=v_j^1+v_j^2+v_j^3 && v_j^1-v_i^1-(1-d_{ji})\lambda_1&\leq 0 \\
   x_{ij}=\lambda_1  && v_j^2-v_i^2+d_{ji}\lambda_2&\leq 0 \\
   x_{ji}=\lambda_2 && v_i^2-v_j^2-(1-d_{ij})\lambda_2&\leq 0 \\
   \lambda_1+\lambda_2+\lambda_3=1 &&v_i^3-v_j^3-(1-d_{ij})\lambda_3&\leq 0 \\
   \lambda_1,\lambda_2,\lambda_3\geq 0 && v_j^3-v_i^3-(1-d_{ji})\lambda_3&\leq 0.
\end{align*}
By projecting out $(\lambda,v_i^3,v_j^3)$, we arrive to

\begin{align*}
 x_{ij}\geq 0,\ x_{ji}\geq 0 &&v_j^2-v_i^2&\leq-d_{ji}x_{ji} \\
 1-x_{ij}-x_{ji}\geq 0&&  v_i^2-v_j^2&\leq (1-d_{ij})x_{ji} \\
 v_i^1-v_j^1\leq -d_{ij}x_{ij}&& u_i-v_i^1-v_i^2-u_j+v_j^1+v_j^2&\leq (1-d_{ij})(1-x_{ij}-x_{ji})  \\
 v_j^1-v_i^1\leq (1-d_{ji})x_{ij}&&  u_j-v_j^1-v_j^2-u_i+v_i^1+v_i^2&\leq (1-d_{ji})(1-x_{ij}-x_{ji}) 
\end{align*}
Along the first three constraints, by projecting out $v_i^1$, we obtain
\begin{align*}
v_j^2-v_i^2\leq-d_{ji}x_{ji}&&u_i-v_i^2-u_j+v_j^2&\leq (1-d_{ij})(1-x_{ij}-x_{ji})-d_{ij}x_{ij}\\
v_i^2-v_j^2\leq (1-d_{ij})x_{ji}&&u_j-v_j^2-u_i+v_i^2&\leq (1-d_{ji})(1-x_{ij}-x_{ji})+(1-d_{ji})x_{ji}.
\end{align*}
By projecting out $v_j^1$, we obtain the result.
\end{proof}

We remark here that the first and second constraints of the convex hull of $Z_\mtz^{ij}(d)$ are precisely the generalized MTZ inequalities \eqref{eq:genMTZ_ineq} associated to arcs $ij$ and $ji\in A_1$. We also note that these constraints are non-redundant and can be strictly satisfied, and thus they define facets of the convex hull of $Z_\mtz^{ij}(d)$.

\subsection{Projecting $Q_\mtz(d)$}
To compare formulations for different choices of $d$, we consider the projection of $Q_\mtz(d)$ on the $x$ variables, that is, the set $P_\mtz(d)=\proj_x(Q_\mtz(d))$.

\begin{proposition}\label{prop:Pmtz(d)}
We have that 
$$P_\mtz(d)=\left\{x\in P_\ap\tq \sum_{ij\in C}x_{ij}\leq |C|-\sum_{ij\in C}d_{ij}\  \textup{for all}\ C\in\mathcal C_1\right\}.$$
\end{proposition}
\begin{proof}
It follows from Lemma \ref{lem:projrequiem} by taking $X=P_\ap$ and defining $\alpha^{ij}$ for each $ij\in A$ by $\alpha^{ij}_{ij}=1$, $\alpha^{ij}_{kl}=0$ for $kl\neq ij$, and $\beta^{ij}=1-d_{ij}$.
\end{proof}

 These inequalities generalize the \emph{Weak circuit inequalities} \eqref{eq:weak_clique_ineq}. Note that $d\in \overline D$ only affects the right-hand side in the inequalities defining $P_\mtz(d)$ and that, since $\sum_{ij\in C}d_{ij}\leq 1$, each constraint in $P_\mtz(d)$ is dominated by the corresponding circuit inequality.


\begin{proposition}\label{prop:facetsPmtz}
For each $d\in\overline D$ and $C\in\mathcal C_1$, the inequality $\sum_{ij\in C}x_{ij}\leq |C|-\sum_{ij\in C}d_{ij}$ defines a facet of $P_\mtz(d)$ if and only if $\sum_{ij\in C}d_{ij}>0$ and $|C|\leq n-2$.
\end{proposition}
\begin{proof}

First observe that given $d\in\overline D$, the vector $\frac{1}{n-1}\ones \in P_\mtz(d)$  satisfies all the inequalities of $P_\mtz(d)$ of the form $\sum_{ij\in C}x_{ij}\leq |C|-\sum_{ij\in C}d_{ij}$ strictly since
$$\sum_{ij\in C}\frac{1}{n-1}=\frac{|C|}{n-1}<|C|-1\leq |C|-\sum_{ij\in C}d_{ij},$$
where the first inequality follows from $n\geq 4$ and $|C|\geq 2$. Therefore, such a constraint defines a facet of $P_\mtz(d)$ if and only if it is non-redundant, that is, removing it from the system defining $P_\mtz(d)$ does add new vectors.

Given $\hat C\in\mathcal C_1$ with $\sum_{ij\in \hat C}d_{ij}>0$ and $|\hat C|\leq n-2$, let $\bar C\subseteq A$ be a directed cycle on the nodes of $N$ not covered by $\hat C$ and consider $\hat x=\ones_{\hat C}+\ones_{\bar C}\in\{0,1\}^A$. We have that $\hat x\in P_\ap$ and satisfies all the constraints in the definition of $P_\mtz(d)$, with the exception of that associated to $\hat C$ as $\sum_{ij\in \hat C}\hat x_{ij}=|\hat C|$. Therefore, this constraint is non-redundant and thus it defines a facet of $P_\mtz(d)$. 

If $\sum_{ij\in C}d_{ij}=0$, then the corresponding inequality is implied by $x_{ij}\leq 1$ for $ij\in C$. If $|C|=n-1$, the corresponding inequality is redundant (see Observation \ref{obs_c1}).
\end{proof}

\subsection{Comparing $P_\mtz(d)$ for different $d\in\overline{D}$}

The following results show that $d$-MTZ formulations defined by vectors in $\interior(\overline{D})$, the topological interior of $\overline{D}$, can be strictly dominated by a $d$-MTZ formulation defined by another vector in $\interior(\overline{D})$, and that, on the other hand, two different $d$-MTZ formulations defined by vectors in the boundary  of  $\overline{D}$ might be incomparable.

\begin{proposition}\label{prop:d_better_d'} The following statements are true:
\begin{enumerate}
\item For all $d\in \textup{int}(\overline{D})$, there exists $d'\in \textup{int}(\overline{D})$ such that $P_\mtz(d')\subsetneq P_\mtz(d)$.
    
    \item Let $C,C'\in\mathcal C_1$ with $|C|,|C'|\leq n-2$, $C\neq C'$ and $d,d'\in \overline{D}$ in the relative interior of the facets of $\overline{D}$ defined by $C,C'$ respectively. Then $P_\mtz(d')$ and $P_\mtz(d)$ are not comparable. 
\end{enumerate}

\end{proposition}
\begin{proof}\text{}
\begin{enumerate}

  \item Let $d\in \textup{int}(\overline{D})$. Then, for all $C\in \mathcal{C}_1$, $\sum_{ij\in C}d_{ij}< 1$.
    Let $\epsilon=\frac{1}{2}\min\{(1-\sum_{ij\in C}d_{ij})/|C|\tq C\in \mathcal{C}_1\}>0$ and define $d'=d+\epsilon \ones$, where $\ones\in \RR^{A_1}$. Observe that for all $C\in \mathcal{C}_1$
    $$\sum_{ij\in C}d'_{ij}= \sum_{ij\in C}d_{ij} +|C|\epsilon < \sum_{ij\in C}d_{ij} +|C|\left(1-\sum_{ij\in C}d_{ij}\right)/|C|=1.$$
    Therefore, $d'\in\textup{int}(\overline{D})$ and, moreover
    \begin{align*}
        P_\mtz(d')&=\{x\in P_{\text{AP}}\tq \sum_{ij\in C}x_{ij}\leq |C|-\sum_{ij\in C}d'_{ij}\quad \textup{for all } C\in\mathcal C_1\}\\
        &\subsetneq \{x\in P_{\text{AP}}\tq \sum_{ij\in C}x_{ij}\leq |C|-\sum_{ij\in C}d_{ij}\quad \textup{for all } C\in\mathcal C_1\}\\
        &=P_\mtz(d),
    \end{align*}

 The strict inclusion follows from the characterizations of the facets of $P_\mtz(d')$ (see Proposition~\ref{prop:facetsPmtz}) and the fact that $\sum_{ij\in C}d_{ij}< \sum_{ij\in C}d'_{ij}$ for all $C\in \mathcal{C}_1$.

    \item Since $d$ is in the relative interior of the facet of $\overline{D}$ defined by $C$, we have that $1=\sum_{ij\in C}d_{ij}>\sum_{ij\in C'}d_{ij}>0$. Similarly, for $d'$ we have that $1=\sum_{ij\in C'}d'_{ij}>\sum_{ij\in C}d'_{ij}>0$. Since $|C|,|C'|\leq n-2$ by Proposition \ref{prop:facetsPmtz} we obtain that the inequalities $\sum_{ij\in C}x_{ij}\leq |C|-1$ and  $\sum_{ij\in C'}x_{ij}\leq |C'|-\sum_{ij\in C'}d_{ij}$ define facets of $P_\mtz(d)$ and that the inequalities $\sum_{ij\in C'}x_{ij}\leq |C'|-1$ and  $\sum_{ij\in C}x_{ij}\leq |C|-\sum_{ij\in C}d'_{ij}$ define facets of $P_\mtz(d)$. This implies that $P_\mtz(d)\nsubseteq P_\mtz(d')$ since any point in the relative interior of the facet $\sum_{ij\in C'}x_{ij}\leq |C'|-\sum_{ij\in C'}d_{ij}$ of $P_\mtz(d)$ it is cut off by the inequality $\sum_{ij\in C'}x_{ij}\leq |C'|-1$ of $P_\mtz(d')$. Similarly, we conclude that $P_\mtz(d')\nsubseteq P_\mtz(d)$, and thus, these two formulations are not comparable. 
\end{enumerate}
\end{proof}

Recall the vector $d^{\mtz}\in D$ defined as $d^{\mtz}_{ij}=\frac{1}{n-1}$ for $ij\in A_1$ which gives the (normalized) MTZ formulation. By Proposition \ref{prop:Pmtz(d)}, we obtain
\begin{align}
P_\mtz(d^{\mtz})&=\left\{x\in P_\ap\tq \sum_{ij\in C}x_{ij}\leq |C|-\sum_{ij\in C}d^{\mtz}_{ij}\  \textup{for all}\ C\in\mathcal C_1\right\}\nonumber\\
&=\left\{x\in P_\ap\tq \sum_{ij\in C}x_{ij}\leq |C|-\frac{|C|}{n-1}\  \textup{for all}\ C\in\mathcal C_1\right\}.\nonumber
\end{align}

The following proposition shows that no anti-symmetric perturbation of $d^{\mtz}$ can improve the associated $d$-MTZ formulation.

\begin{proposition}
   Let $\delta\in\RR^{A_1}$ be such that $d^{\mtz}+\delta\in D$ and $\delta_{ij}=-\delta_{ji}$ for any $ij\in A_1$. Then $P_\mtz(d^{\mtz}+\delta)\subseteq P_\mtz(d^{\mtz})$ implies $P_\mtz(d^{\mtz}+\delta)=P_\mtz(d^{\mtz})$.
\end{proposition}
\begin{proof}
First, since $d^{\mtz}+\delta\in D$ we have that for any $C\in \mathcal{C}_1$ the following inequality holds 
\begin{equation}
    \sum_{ij\in C}\left(d_{ij}^{\mtz}+\delta_{ij}\right)\leq 1\label{in_assumption}
\end{equation} 

On the other hand, since $d^{\mtz},d^{\mtz}+\delta\in D$ by the characterizations of $P_\mtz(d)$ (Proposition~\ref{prop:Pmtz(d)}) and their facets (Proposition 
~\ref{prop:facetsPmtz}) we have that  $P_\mtz(d^{\mtz}+\delta)\subseteq P_\mtz(d^{\mtz})$ if and only if 
 $$|C|-\sum_{ij\in C} \left(d^{\mtz}_{ij}+\delta_{ij}\right)\leq|C|-\sum_{ij\in C}d^{\mtz}_{ij}$$ 
\noindent for any $C\in \mathcal{C}_1$ such that $|C|\leq n-2$, as the facets of $P_\mtz(d^{\mtz})$ are defined by the inequalities $\sum_{ij\in C}x_{ij}\leq |C|-\sum_{ij\in C}d^{\mtz}_{ij}$, and similarly for $P_\mtz(d^{\mtz}+\delta)$. This condition is equivalent to
\begin{equation}
\sum_{ij\in C}\delta_{ij}\geq0\label{subset_assumption}
\end{equation}
\noindent for any $C\in \mathcal{C}_1$ such that $|C|\leq n-2$. 

From this discussion, it follows that $P_\mtz(d^{\mtz}+\delta)=P_\mtz(d^{\mtz})$ if and only if we have that
$\sum_{ij\in C}\delta_{ij}=0$ for any $C\in \mathcal{C}_1$ with $|C|\leq n-2$. We will use this equivalence later in the proof.

Observe that by the anti-symmetry assumption on $\delta$, that is, $\delta_{ij}=-\delta_{ji}$ for any $ij\in A_1$, we have that $\sum_{ij\in  C}\delta_{ij}=0$ for any $ C \in \mathcal{C}_1$ with $| C|=2$. Also, by \eqref{in_assumption}, note that any cycle $ C \in \mathcal{C}_1$ with $| C|=n-1$ satisfies $\sum_{ij\in  C}\delta_{ij}\leq 1-\frac{| C|}{n-1}=0$.

In order to show that $\sum_{ij\in  C}\delta_{ij}=0$ for any cycle $ C \in \mathcal{C}_1$ with $2<| C|<n-1$, we will show that given $m$ with $3<m\leq n-1$, if each $ C \in \mathcal{C}_1$ with $| C|=m$ satisfies $\sum_{ij\in  C}\delta_{ij}=0$, then any $ C \in \mathcal{C}_1$ with $| C|=m-1$ also satisfies $\sum_{ij\in  C}\delta_{ij}=0$. Let $C\in\mathcal C_1$ with $| C|=m-1$. Since $3<m\leq n-1$, there exist $lk\in A_1$ an edge of $ C$ and $v\in N_1$ a node not in $ C$ such that the cycle $C'= C\setminus \{lk\} \cup \{lv,vk\}$ has size $|C'|=m$. By using inequalities \eqref{subset_assumption} for the cycles $\{kl,lv,vk\}$ and $ C$, and the assumption that the property is true for any cycle of size $m$, we obtain that
\begin{equation}
\delta_{kl}+\delta_{lv}+\delta_{vk}\geq0, \quad
\sum_{ij\in  C}\delta_{ij}\geq0, \quad
\sum_{ij\in C'}\delta_{ij}=0.    \label{induction_hypothesis}
\end{equation}
Moreover, by definition of $C'$ we can compute
 $$\sum_{ij\in C'}\delta_{ij}=
    \sum_{ij\in  C}\delta_{ij}-\delta_{lk}+\delta_{lv}+\delta_{vk}.$$
Since $\delta_{lk}=-\delta_{kl}$ by the anti-symmetric assumption and \eqref{induction_hypothesis} we obtain that
$$0= \sum_{ij\in C'}\delta_{ij}=\sum_{ij\in  C}\delta_{ij}+(\delta_{kl}+\delta_{lv}+\delta_{vk}).$$
As the two terms in the sum are non-negative, we conclude $\sum_{ij\in  C}\delta_{ij}=0.$ This shows that $P_\mtz(d^{\mtz}+\delta)=P_\mtz(d^{\mtz})$, as desired.

\end{proof}

\subsection{The intersection of all $d$-MTZ formulations} \label{sec:closure_d-MTZformulations}

Given nonempty $D'\subseteq \overline D$, recall the set
$$\Ef(Q_\mtz(D'))=\left\{(x,u)\in P_\ap\times\left(\RR^{N_1}\right)^{D'}\tq (x,u(d))\in Q_\mtz(d)\ \forall d\in D'\right\}.$$

\begin{proposition}\label{prod_mtz}
    Let $D'\subseteq \overline D$. If for each $C\in\mathcal C_1$, there exists $d\in D'$ such that $\sum_{ij\in C}d_{ij}>0$, then no integer $x\in \proj_x(\Ef(Q_\mtz(D')))$ can define subtours. 
    In particular, if $D'$ is finite, then $\Ef(Q_\mtz(D'))$ is a valid extended formulation for the ATSP.
\end{proposition}

\begin{proof}
    We have that $\proj_x(\Ef(Q_\mtz(D')))=\bigcap_{d\in D'}\proj_x(Q_\mtz(d))$. Therefore, by Proposition \ref{prop:Pmtz(d)}, for any $x\in\proj_x(\Ef(Q_\mtz(D')))$ and $C\in\mathcal C_1$, it holds that $\sum_{ij\in C}x_{ij}\leq |C|-\sum_{ij\in C}d_{ij}$ for all $d\in D'$, or equivalently, $\sum_{ij\in C}x_{ij}\leq |C|-\max\{\sum_{ij\in C}d_{ij}\tq d\in D'\}$. By hypothesis, the right-hand side is at least $|C|-1$, but less than $|C|$. Thus, no integer $x\in \proj_x(\Ef (Q_\mtz(D')))$ can define subtours and the last assertion follows.
\end{proof}

Note that any $D'=\{d\}$ with $d\in D$ satisfies the above condition since $d>0$, in which case $Q_\mtz(D')=Q_\mtz(d)$. Also, observe that we can choose $D'$ contained in the boundary of $\overline D$ as long it satisfies the condition in Proposition~\ref{prod_mtz}. Following the above result, for each $k\in N_1$, let $d^k\in \RR^{A_1}_{+}$ be such that
$$d^k_{ij}=\left\{\begin{array}{cc}1&i=k\\ 0&i\neq k, \end{array}\right.$$
and define $V_\mtz=\{d^k\tq k\in N_1\}$. Note that for each $k\in N_1$ and $C\in \mathcal{C}_1$, we have $\sum_{ij\in C}d^k_{ij}=1$ if $kj\in C$ for some $j\in N_1$ and $\sum_{ij\in C}d^k_{ij}=0$ else. Therefore, $V_\mtz\subseteq \overline D$. Moreover, each element of $V_\mtz$ is a vertex of $\overline D$, although they do not encompass the complete set of vertices of $\overline D$ in general.

\begin{theorem}\label{cl_mtz}
We have that $\Cl(P_\mtz(D))=\Cl(P_\mtz(V_\mtz))=\overline P_\mtz$, where
 $$\overline P_\mtz = \left\{x\in P_{\text{AP}}\tq \sum_{ij\in C}x_{ij}\leq |C|-1\ \textup{for all } C\in\mathcal C_1\right\}.$$
\end{theorem}
\begin{proof}
Noting that $\overline P_\mtz\subseteq P_\mtz(d)$ for all $d\in D$, we have $\overline P_\mtz\subseteq \Cl(P_\mtz(D))=\Cl(P_\mtz(\overline D))\subseteq \Cl(P_\mtz(V_\mtz))$. Now, to show that $\Cl(P_\mtz(V_\mtz))=\overline P_\mtz$, observe that from Proposition~\ref{prop:Pmtz(d)}, for each $k\in N_1$ we have
\begin{align*}
P_\mtz(d^k)&=\left\{x\in P_{\text{AP}}\tq \sum_{ij\in C}x_{ij}\leq |C|-\sum_{ij\in C}d^k_{ij}\ \textup{for all } C\in\mathcal C_1\right\}\\
&=\left\{x\in P_{\text{AP}}\tq \begin{array}{rl}
\displaystyle\sum_{ij\in C}x_{ij}\leq |C|-1\ \textup{for all } C\in\mathcal C_1\tq C\cap\delta^+(k)\neq\emptyset\\
\displaystyle\sum_{ij\in C}x_{ij}\leq |C|\ \textup{for all } C\in\mathcal C_1\tq C\cap\delta^+(k)=\emptyset
\end{array}\right\}\\
&=\left\{x\in P_{\text{AP}}\tq \sum_{ij\in C}x_{ij}\leq |C|-1\ \textup{for all } C\in\mathcal C_1\tq C\cap\delta^+(k)\neq\emptyset\right\}.
\end{align*}
Therefore,
\begin{align*}
\Cl(P_\mtz(V_\mtz))&=\bigcap_{k\in N_1}P_\mtz(d^k)\\
&=\left\{x\in P_{\text{AP}}\tq \sum_{ij\in C}x_{ij}\leq |C|-1\ \textup{for all } C\in\mathcal C_1,\ \textup{for all}\ k\in N_1\tq C\cap\delta^+(k)\neq\emptyset\right\}\\
&=\left\{x\in P_{\text{AP}}\tq \sum_{ij\in C}x_{ij}\leq |C|-1\ \textup{for all } C\in\mathcal C_1\right\}=\overline P_\mtz.
\end{align*}
\end{proof}
We remark here that from Theorem~\ref{cl_mtz}, $\Cl(P_\mtz(D))$ can be obtained with $\mathcal O(n)$ vectors in $\overline D$ which are easy to identify, instead of requiring the complete set of vertices of $\overline D$ as in Corollary~\ref{prod_V}. Moreover, the latter is unlikely to be efficiently accomplished in view of Proposition~\ref{nphard}. 

Also, observe that $\Cl(P_\mtz(D))$ is given by the circuit inequalities \eqref{eq:circuit_ineq}, and therefore it is directly related to the RMTZ formulation. Corollary~\ref{cor_mtz} below recovers Theorem~1 in \cite{gouveia1999asymmetric}.

 \begin{corollary}\label{cor_mtz}
 We have that $\Cl(P_\mtz(D))=\proj_x(\overline Q_\mtz)$, where
$$\overline Q_\mtz=\left\{(x,(v^k\tq k\in N_1))\in P_\ap\times\RR^{N_1\times N_1}\tq \begin{array}{rl}
\displaystyle v^k_k-v^k_j\leq -x_{kj}& \textup{for all}\ k,j\in N_1\tq j\neq k\\
\displaystyle v^k_i-v^k_j\leq 1-x_{ij}& \textup{for all}\ k,i,j\in N_1\tq  i\neq k,\ j\neq i
\end{array}\right\}.$$
\end{corollary}

\begin{proof}
Noting that $\overline Q_\mtz=\Ef\left(P_\mtz(V_\mtz)\right)$, the result follows from Theorem~\ref{cl_mtz} and Corollary~\ref{prod_V}.
 \end{proof}



\section{Parametric $d$-DL formulations}\label{sec:dl}

For $d\in\RR^{A_1}_{++}$, consider the generalization of the DL formulation
\begin{equation}\label{eq:genDL_ineq_M}
    u_i-u_j+Mx_{ij}+(M-d_{ij}-d_{ji})x_{ji}\leq M-d_{ij}\quad \textup{for all}\ ij\in A_1.
\end{equation}
As with the MTZ formulation, for large enough $M$, inequalities \eqref{eq:genDL_ineq_M} define a valid extended formulation for the ATSP that we call $d$-DL formulation: for $x_{ij}=0$ and $x_{ji}=0$, \eqref{eq:genDL_ineq_M} for $ij$ and $ji$ reduces to the superfluous constraints $u_i-u_j+d_{ij}\leq M$ and $u_i-u_j+d_{ji}\leq M$, respectively, whereas for $x_{ij}=1$ and $x_{ji}=0$, it implies $u_j=u_i + d_{ij}>u_i$ and thus prohibits $x$ from defining cycles in $\mathcal C_1$.
Normalizing with $M=1$, we obtain~\eqref{eq:genDL_ineq}:
$$u_i-u_j+x_{ij}+(1-d_{ij}-d_{ji})x_{ji}\leq 1-d_{ij}\quad \textup{for all}\ ij\in A_1.$$

For $d\in \overline D$, recall the set
$$Q_\dl(d)=\left\{(x,u)\in P_\ap\times\RR^{N_1}\tq u_i-u_j+x_{ij}+(1-d_{ij}-d_{ji})x_{ji}\leq 1-d_{ij}\ \text{for all }ij\in A_1\right\}.$$ 
Note that since $x\geq0$ and $d_{ij}+d_{ji}\leq 1$ for any $d\in\overline D$, we have that \eqref{eq:genDL_ineq} implies \eqref{eq:genMTZ_ineq} and thus $Q_\dl(d)\subseteq Q_\mtz(d)$.

\subsection{Obtaining the $d$-DL formulation from convexification}
As with the MTZ formulation, we first argue that our generalization of the DL formulation is sound.

For $d\in D$ and $ij\in A_1$, let $Z^{ij}_\dl(d)$ be the set of vectors $(u_i,u_j,x_{ij},x_{ji})\in\RR^2\times\{0,1\}^2$ such that $x_{ij}+x_{ji}\leq 1$ and
\begin{itemize}
    \item $[x_{ij}=1,\ x_{ji}=0]\Rightarrow u_j= u_i+d_{ij}$
    \item $[x_{ij}=0,\ x_{ji}=1]\Rightarrow u_i= u_j+d_{ji}$
    \item $[x_{ij}=0,\ x_{ji}=0]\Rightarrow u_i-u_j\leq 1-d_{ij},\ u_j-u_i\leq 1-d_{ji}$.
\end{itemize}

\begin{proposition}
The convex hull of $Z^{ij}_\dl(d)$ is given by
\begin{eqnarray*}
u_i-u_j+x_{ij}+(1-d_{ij}-d_{ji})x_{ji}\leq 1-d_{ji}\\
u_j-u_i+x_{ji}+(1-d_{ji}-d_{ij})x_{ij}\leq 1-d_{ij}\\
x_{ij}\geq 0\\
x_{ji}\geq 0.
\end{eqnarray*}
\end{proposition}
Observe that the constraint $x_{ij}+x_{ji}\leq 1$ is redundant as it is implied by the first two, which are indeed the generalized DL inequalities.
\begin{proof}
Let us write $Z^{ij}_\dl(d)$ as the union of the following sets:
\begin{itemize}
    \item $\{(u_i,u_j)\in\RR^2\tq u_i+d_{ij}= u_j\}\times\{(1,0)\}$
    \item $\{(u_i,u_j)\in\RR^2\tq u_j+d_{ji}= u_i\}\times\{(0,1)\}$
    \item $\{(u_i,u_j)\in\RR^2\tq u_i-u_j\leq 1-d_{ij},\ u_j-u_i\leq 1-d_{ji}\}\times\{(0,0)\}$
\end{itemize}
By disjunctive programming \cite{balas1979disjunctive}, the convex hull of $Z^{ij}_\dl(d)$ is given by the projection on the $(u_i,u_j,x_{ij},x_{ji})$ variables of the system

\begin{align*}
u_i=v_1^1+v_i^2+v_i^3&&v_i^1-v_j^1+d_{ij}\lambda_1&= 0 \\
u_j=v_j^1+v_j^2+v_j^3&& v_j^2-v_i^2+d_{ji}\lambda_2&= 0\\
x_{ij}=\lambda_1&& v_i^3-v_j^3-(1-d_{ij})\lambda_3&\leq 0\\
x_{ji}=\lambda_2&&v_j^3-v_i^3-(1-d_{ji})\lambda_3&\leq 0. \\
\lambda_1+\lambda_2+\lambda_3=1 && \\
\lambda_1,\lambda_2,\lambda_3\geq 0&&
\end{align*}

By projecting out $(\lambda,v_i^3,v_j^3,v_j^1,v_i^2)$, we obtain the result.
\end{proof}

\subsection{Projecting $Q_\dl(d)$}
We now turn our attention to the set $P_\dl(d)=\proj_x(Q_\dl(d))$.

\begin{proposition}\label{prop:Pdl(d)}
We have that 
$$P_\dl(d)=\left\{x\in P_{\text{AP}}: \begin{array}{rl}\sum_{ij\in C}(x_{ij}+x_{ji})-\sum_{ij\in C}(d_{ij}+d_{ji})x_{ji}\leq |C|-\sum_{ij\in C}d_{ij}&\textup{for all}\ C\in\mathcal C_1\tq |C|\geq 3\\
x_{ij}+x_{ji}\leq 1 &\textup{for all}\ ij\in A_1\end{array}\right\}.$$
\end{proposition}
\begin{proof}
From Lemma \ref{lem:projrequiem}, by taking $X=P_\ap$ and defining $\alpha^{ij}$ for each $ij\in A$ by $\alpha^{ij}_{ij}=1$, $\alpha^{ij}_{ji}=1-d_{ij}-d_{ji}$, $\alpha^{ij}_{kl}=0$ for $kl\neq ij,ji$, and $\beta^{ij}=1-d_{ij}$, we obtain
$$P_\dl(d)=\left\{x\in P_{\text{AP}}\tq \sum_{ij\in C}(x_{ij}+x_{ji})-\sum_{ij\in C}(d_{ij}+d_{ji})x_{ji}\leq |C|-\sum_{ij\in C}d_{ij}\  \textup{for all } C\in\mathcal C_1\right\}.$$
Observe that if $|C|=2$, then $C=\{ij,ji\}$ for some $ij\in A_1$, and the corresponding inequality reduces to
$$(x_{ij}+x_{ji})+(x_{ji}+x_{ij})-(d_{ij}+d_{ji})x_{ji}-(d_{ji}+d_{ij})x_{ij}\leq 2-d_{ij}-d_{ji}.$$
Dividing both sides by $2-d_{ij}-d_{ji}>0$, the result follows.
\end{proof}

Note that $d$ affects both left- and right-hand sides of $P_\dl(d)$, in contrast to what happens with $P_\mtz(d)$ where $d$ appears in the right-hand side only. Also, note that the inequalities $x_{ij}+x_{ji}\leq 1$ for $ij\in A_1$ are independent of $d$. Finally, observe that $P_\dl(d)\subseteq P_\mtz(d)$ for all $d\in \overline D$ as $Q_\dl(d)\subseteq Q_\mtz(d)$.


\begin{proposition}\label{prop:facetsPdl}
For each $d\in\overline D$ and $C\in\mathcal C_1$, the inequality $\sum_{ij\in C}(x_{ij}+x_{ji})-\sum_{ij\in C}(d_{ij}+d_{ji})x_{ji}\leq |C|-\sum_{ij\in C}d_{ij}$ defines a facet of $P_\dl(d)$ if and only if $|C|=2$ or $\sum_{ij\in C}d_{ij}>0$ and $3\leq|C|\leq n-2$.
\end{proposition}
\begin{proof}
Given $\hat C\in\mathcal C_1$ with $|\hat C|=2$ or $\sum_{ij\in \hat C}d_{ij}>0$ and $3\leq|\hat C|\leq n-2$, let $\bar C\subseteq A$ be a directed cycle on the nodes of $N$ not covered by $\hat C$ and consider $\hat x=\ones_{\hat C}+\ones_{\bar C}\in\{0,1\}^A$. 
We have that $\hat x\in P_\ap$ and satisfies all the constraints in the definition of $P_\dl(d)$, with the exception of that associated to $\hat C$ as
$\sum_{ij\in\hat  C}(\hat x_{ij}+\hat x_{ji})-\sum_{ij\in \hat C}(d_{ij}+d_{ji})\hat x_{ji}=|\hat C|.$ 
Therefore this constraint is non-redundant. This inequality is satisfied strictly by $\frac{1}{n-1}\ones \in P_\dl(d)$ since, recalling that $n\geq 4$, we have $\frac{|\hat C|}{n-1}<|\hat C|-1$ if $|\hat C|=2$ and, if $|\hat C|\geq 3$ we have
$$\sum_{ij\in \hat C}\frac{2}{n-1}-\sum_{ij\in \hat C}(d_{ij}+d_{ji})\frac{1}{n-1}<|\hat C|\frac{2}{n-1}\leq |\hat C|\frac{2}{3}\leq |\hat C|-1\leq |\hat C|-\sum_{ij\in \hat  C}d_{ij}.$$
We conclude that the constraint defines a facet of $P_\dl(d)$. 

If $\sum_{ij\in C}d_{ij}=0$, then the corresponding inequality is implied by $x_{ij}+x_{ji}\leq 1$ for $ij\in C$ and $-d_{ji}x_{ji}\leq 0$ for $ij\in C$. If $|C|=n-1$, the corresponding inequality is redundant (see Observation \ref{obs_c1}).
\end{proof}

\subsection{Comparing $P_\dl(d)$ for different $d\in \overline{D}$}
Unlike the case of $P_\mtz(d)$, by Proposition \ref{prop:Pdl(d)}, the parameter $d$ in the inequalities defining $P_\dl(d)$ not only appears as a constant in the r.h.s., but also in the l.h.s. multiplying some of the variables. As a consequence, the comparison of different $d$-DL formulations $P_\dl(d)$ for different vectors $d$ is not as straightforward as in the case of $d$-MTZ formulations. The next result shows that under some technical assumptions, anti-symmetric perturbations yield incomparable formulations.

\begin{proposition}
Let $d\in D$ and $\delta \in \RR^{A_1}\setminus\{0\}$ be such that $d+\delta\in D$, $\delta_{ij}=-\delta_{ji}$ for all $ij \in A_1$, and assume that there exists $\hat C \in\mathcal C_1$ with $|\hat C|\geq 3$ such that $\sum_{ij\in \hat C}\delta_{ij}\neq0$. Then $P_{\dl}(d)$ and $P_{\dl}(d+\delta)$ are not comparable, that is, none of these formulations is included in the other.  
\end{proposition}
\begin{proof} Since $d\in D$ we have that $\sum_{ij\in C}d_{ij}>0$ for any  $C\in\mathcal C_1$. Then by Proposition~\ref{prop:facetsPdl}, the non-trivial facets of $P_\dl(d)$ are given by the inequalities
\begin{equation}\label{cycle_ineq_d}
\sum_{ij\in C}(x_{ij}+x_{ji})-\sum_{ij\in C}(d_{ij}+d_{ji})x_{ji}\leq |C|-\sum_{ij\in C}d_{ij}
\end{equation}
\noindent for all $C\in\mathcal C_1$ with $|C|\leq n-2$.
Given $C\in\mathcal C_1$ with $|C|\leq n-2$, let $\revcycle{C}=\{ji\tq ij\in C\}$ be its associated reversed cycle. We can write the facet of $P_\dl(d)$ associated to $\revcycle{C}$ as follows:
 $$\sum_{ij\in \revcycle{C}}(x_{ij}+x_{ji})-\sum_{ij\in \revcycle{C}}(d_{ij}+d_{ji})x_{ji}\leq |\revcycle{C}|-\sum_{ij\in \revcycle{C}}d_{ij}.$$
Equivalently, we can write
\begin{equation}\label{reversed_cycle_ineq_d}
\sum_{ij\in C}(x_{ij}+x_{ji})-\sum_{ij\in C}(d_{ij}+d_{ji})x_{ij}\leq |C|-\sum_{ij\in C}d_{ji}.
\end{equation}
Now, notice that 
$$d_{ij}+d_{ji}= (d_{ij}+\delta_{ij})+(d_{ji}-\delta_{ij})=(d_{ij}+\delta_{ij})+(d_{ji}+\delta_{ji})=(d+\delta)_{ij}+(d+\delta)_{ji},$$
where the second equality is given by the anti-symmetry assumption on $\delta$, that is, $\delta_{ij}=-\delta_{ji}$ for any $ij\in A_1$.
Since $d+\delta\in D$, this implies that the non-trivial facets of $P_\dl(d+\delta)$ are given by the inequalities
\begin{equation}\label{cycle_ineq_d_plus_delta}
\sum_{ij\in C}(x_{ij}+x_{ji})-\sum_{ij\in C}(d_{ij}+d_{ji})x_{ji}\leq |C|-\sum_{ij\in C}d_{ij}-\sum_{ij\in C}\delta_{ij}
\end{equation}
for all $C\in\mathcal C_1$ with $|C|\leq n-2$.
For the reversed cycle we obtain:
$$\sum_{ij\in C}(x_{ij}+x_{ji})-\sum_{ij\in C}(d_{ij}+d_{ji})x_{ij}\leq |C|-\sum_{ij\in C}d_{ji}-\sum_{ij\in C}\delta_{ji}.$$
And by the anti-symmetry assumption on $\delta$ for $\revcycle{C}$ we can write
\begin{equation}\label{reversed_cycle_ineq_d_plus_delta}
\sum_{ij\in C}(x_{ij}+x_{ji})-\sum_{ij\in C}(d_{ij}+d_{ji})x_{ij}\leq |C|-\sum_{ij\in C}d_{ji}+\sum_{ij\in C}\delta_{ij}.
\end{equation}

Let $\hat C \in\mathcal C_1$ as in the statement of the proposition and notice that without loss of generality we may assume that $\sum_{ij\in \hat C}\delta_{ij}>0$. Then by Proposition \ref{prop:facetsPdl} the inequalities \eqref{cycle_ineq_d} and \eqref{reversed_cycle_ineq_d} of $P_\dl(d)$ and the inequalities \eqref{cycle_ineq_d_plus_delta} and \eqref{reversed_cycle_ineq_d_plus_delta} of $P_\dl(d+\delta)$ associated to $\hat C$ are all facet defining. Now, since $\sum_{ij\in \hat C}\delta_{ij}>0$,  for the cycle $\hat C$ we have that the r.h.s. of inequality \eqref{cycle_ineq_d_plus_delta} of $P_{\dl}(\hat d+\delta)$ is strictly less than the r.h.s. of  inequality \eqref{cycle_ineq_d} of $P_{\dl}(\hat d)$. As both inequalities define facets, we obtain that $P_{\dl}(\hat d)$ is not contained in $P_{\dl}(\hat d+\delta)$. Similarly, for the cycle $\hat C$ the r.h.s. of inequality \eqref{reversed_cycle_ineq_d} of $P_{\dl}(\hat d) $ is strictly less than the r.h.s. of inequality \eqref{reversed_cycle_ineq_d_plus_delta} of $P_{\dl}(\hat d+\delta)$. Therefore, we obtain that $P_{\dl}(\hat d)$ is not contained in $P_{\dl}(\hat d+\delta)$. We conclude that $P_{\dl}(\hat d)$ and $P_{\dl}(\hat d+\delta)$ are not comparable. 
\end{proof}

We note here that the assumption ``there exists $\hat C \in\mathcal C_1$ with $|\hat C|\geq 3$ such that $\sum_{ij\in \hat C}\delta_{ij}\neq0$'' is not superfluous as the following example shows: consider $N_1=\{2,3,4\}$ and $\delta_{32}=-\delta_{23}=1$, $\delta_{34}=-\delta_{43}=1/2$, and $\delta_{42}=-\delta_{24}=1/2$. Then $\delta$ satisfies the anti-symmetric property and $\sum_{ij\in C}\delta_{ij}=0$ for any cycle $C \in\mathcal C_1$.

\subsection{The intersection of all $d$-DL formulations}\label{sec:closure_d-DLformulations}

Given nonempty $D'\subseteq \overline D$, consider
$$\Ef(Q_\dl(D'))=\left\{(x,u)\in P_\ap\times\left(\RR^{N_1}\right)^{D'}\tq (x,u(d))\in Q_\dl(d)\ \forall d\in D'\right\}.$$

\begin{proposition}\label{prod_dl}
    Let $D'\subseteq D$. If for each $C\in\mathcal C_1$, there exists $d\in D'$ such that $\sum_{ij\in C}d_{ij}>0$, then no integer $x\in \proj_x(\Ef(Q_\dl(D')))$ can define subtours. 
    In particular, if $D'$ is finite, then $\Ef(Q_\dl(D'))$ is a valid extended formulation for the ATSP.
\end{proposition}

\begin{proof}
    The result follows from Proposition~\ref{prod_mtz} by noting that $\proj_x(\Ef(Q_\dl(D'))) \subseteq \proj_x(\Ef(Q_\mtz(D')))$.
\end{proof}

As with the MTZ formulation, any $D'=\{d\}$ with $d\in D$ satisfies the above condition since $d>0$, in which case $Q_\dl(D')=Q_\dl(d)$. Also, we can choose $D'$ contained in the boundary of $\overline D$ as long it satisfies the conditions if Proposition~\ref{prod_dl}. Following the above result, for $kl\in A_1$, let $d^{kl}$ be the $kl$-th canonical vector in $\RR^{A_1}$, and define $V_\dl=\{d^{kl}\tq kl \in A_1\}$. Clearly, we have $V_\dl\subseteq \overline D$ and each element of $V_\dl$ is a vertex of $\overline D$.

\begin{theorem}\label{cl_dl}
We have that $\Cl(P_\dl(D))=\Cl(P_\dl(V_\dl))=\overline P_\dl$, where
$$\overline P_\dl=\left\{x\in P_\ap\tq 
\begin{array}{rl}\sum_{ij\in C}(x_{ij}+x_{ji})-x_{lk}\leq |C|-1& \textup{for all } C\in\mathcal C_1,\ \textup{for all } kl\in C\tq |C|\geq 3\\x_{ij}+x_{ji}\leq 1& \textup{for all } ij\in A_1
\end{array}\right\}.$$
\end{theorem}

\begin{proof}
From Proposition~\ref{prop:Pdl(d)}, for any $kl \in A_1$ we have
$$P_\dl(d^{kl})
=\left\{x\in P_\ap\tq 
\begin{array}{rl}
\sum_{ij\in C}(x_{ij}+x_{ji})-\sum_{ij\in C}(d^{kl}_{ij}+d^{kl}_{ji})x_{ji}\leq |C|-\sum_{ij\in C}d^{kl}_{ij}&  \textup{for all }C\in\mathcal C_1\tq |C|\geq 3\\
x_{ij}+x_{ji}\leq 1&\textup{for all } ij\in A_1
\end{array}
\right\}$$
$$=\left\{x\in P_\ap\tq 
\begin{array}{rl}
\sum_{ij\in C}(x_{ij}+x_{ji})-x_{lk}\leq |C|-1&  \textup{for all } C\in\mathcal C_1\tq kl\in C,\ |C|\geq 3\\
\sum_{ij\in C}(x_{ij}+x_{ji})-x_{kl}\leq |C|&  \textup{for all } C\in\mathcal C_1\tq lk\in C,\ |C|\geq 3\\
\sum_{ij\in C}(x_{ij}+x_{ji})\leq |C|&  \textup{for all } C\in\mathcal C_1\tq kl,lk\notin C,\ |C|\geq 3\\
x_{ij}+x_{ji}\leq 1&\textup{for all } ij\in A_1
\end{array}
\right\}$$
$$=\left\{x\in P_\ap\tq 
\begin{array}{rl}
\sum_{ij\in C}(x_{ij}+x_{ji})-x_{lk}\leq |C|-1&  \textup{for all } C\in\mathcal C_1\tq kl\in C,\ |C|\geq 3\\
x_{ij}+x_{ji}\leq 1&\textup{for all } ij\in A_1
\end{array}
\right\}.$$
Therefore,
\begin{align*}
\Cl(P_\dl(V_\dl))&=\bigcap_{kl\in A_1}P_\dl(d^{kl})\\
&=\left\{x\in P_\ap\tq 
\begin{array}{rl}
\sum_{ij\in C}(x_{ij}+x_{ji})-x_{lk}\leq |C|-1&  \forall C\in\mathcal C_1,\ \forall kl\in C,\ |C|\geq 3\\
x_{ij}+x_{ji}\leq 1&\forall ij\in A_1
\end{array}
\right\}=\overline P_\dl.
\end{align*}

We thus obtain $\Cl(P_\dl(D))=\Cl(P_\dl(\overline D))\subseteq \Cl(P_\dl(V_\dl))=\overline P_\dl$. To show that $\overline P_\dl\subseteq \Cl(P_\dl(D))$, let $x\in \overline P_\dl$, $d\in D$ and $C\in\mathcal C_1$ with $|C|\geq 3$. We have
$$\sum_{ij\in C}(x_{ij}+x_{ji})+1-x_{lk}\leq |C|\quad \forall kl\in C.$$
Multiplying each of these inequalities by $\frac{d_{kl}}{\sum_{kl\in C}d_{kl}}$ and summing over all $kl \in C$, we obtain
$$\sum_{ij\in C}(x_{ij}+x_{ji})+\frac{1}{\sum_{kl\in C}d_{kl}}\sum_{ij\in C}(d_{ij}-d_{ij}x_{ji})\leq |C|.$$
Since $\sum_{kl\in C}d_{kl}\leq 1$ and $d_{ij}-d_{ij}x_{ji}\geq0$, we have
$$\sum_{ij\in C}(x_{ij}+x_{ji})+\sum_{ij\in C}(d_{ij}-d_{ij}x_{ji})\leq |C|.$$
As $\sum_{ij\in C}-d_{ji}x_{ji}\leq 0$, the latter inequality implies that
$$\sum_{ij\in C}(x_{ij}+x_{ji})+\sum_{ij\in C}(d_{ij}-(d_{ij}+d_{ji})x_{ji})\leq |C|.$$
Therefore, $x\in P_\dl(d)$ for all $d\in D$ and thus $\overline P_\dl \subseteq \Cl(P_\dl(D))$.
\end{proof}

 \begin{corollary}
 We have that $\Cl(P_\dl(D))=\proj_x(\overline Q_\dl)$, where
$$\overline Q_\dl=\left\{(x,(u^{kl}\tq kl\in A_1))\in P_\ap\times\RR^{A_1\times N_1}_+\tq \begin{array}{rl}
\displaystyle u^{kl}_k-u^{kl}_l\leq -x_{kl}& \textup{for all}\ kl\in A_1\\
\displaystyle u^{kl}_l-u^{kl}_k\leq 1-x_{lk}& \textup{for all}\  kl\in A_1\\
\displaystyle u^{kl}_i-u^{kl}_j\leq 1-x_{ij}-x_{ji}& \textup{for all} \ kl,ij\in A_1\tq ij\neq kl,\ ij\neq lk
\end{array}\right\}.$$
\end{corollary}

\begin{proof}
Noting that $\overline Q_\dl=\Ef(Q_\dl(V_\dl))$, the result follows from Theorem~\ref{cl_dl} and Corollary~\ref{prod_V}.
\end{proof}
Theorem~\ref{cl_dl} shows that $\Cl(P_\dl(D))$ can be obtained with a modest number of vectors in $\overline D$ which are easy to identify, instead of the complete set of vertices of $\overline D$. This time, however, we require $\mathcal O(n^2)$ such vectors instead of $\mathcal O(n)$ as with the MTZ formulation. Moreover, it can be shown that $V_\mtz$ does not necessarily yield $\Cl(P_\dl(D))$.

\begin{proposition}\label{dl_mtz}
 $\Cl(P_\dl(V_\mtz))$ is given by
$$\left\{x\in P_\ap\tq 
\begin{array}{rl}
\sum_{ij\in C}(x_{ij}+x_{ji})-x_{k k^C_-}-x_{k^C_+ k}\leq |C|-1&  \textup{for all}\ C\in\mathcal C_1,\ \textup{for all}\ k\in C\tq |C|\geq 3\\
x_{ij}+x_{ji}\leq 1&\textup{for all}\ ij\in A_1
\end{array}
\right\},$$
where $k^C_-$ and $k^C_+$ denote the node preceding and succeeding $k$ in cycle $C$, respectively.
\end{proposition}
\begin{proof}
From Proposition~\ref{prop:Pdl(d)}, for any $k \in N_1$ we have
$$P_\dl(d^k)
=\left\{x\in P_\ap\tq 
\begin{array}{rl}
\sum_{ij\in C}(x_{ij}+x_{ji})-\sum_{ij\in C}(d^k_{ij}+d^k_{ji})x_{ji}\leq |C|-\sum_{ij\in C}d^k_{ij}&  \textup{for all }C\in\mathcal C_1\tq |C|\geq 3\\
x_{ij}+x_{ji}\leq 1&\textup{for all } ij\in A_1
\end{array}
\right\}$$
$$=\left\{x\in P_\ap\tq 
\begin{array}{rl}
\sum_{ij\in C}(x_{ij}+x_{ji})-\sum_{ij\in C}(d^k_{ij}+d^k_{ji})x_{ji}\leq |C|-1&  \textup{for all }C\in\mathcal C_1\tq |C|\geq 3,\ k\in C\\
\sum_{ij\in C}(x_{ij}+x_{ji})\leq |C|&  \textup{for all }C\in\mathcal C_1\tq |C|\geq 3,\ k\notin C\\
x_{ij}+x_{ji}\leq 1&\textup{for all } ij\in A_1
\end{array}
\right\}$$
$$=\left\{x\in P_\ap\tq 
\begin{array}{rl}
\sum_{ij\in C}(x_{ij}+x_{ji})-x_{k k^C_-}-x_{k^C_+ k}\leq |C|-1&  \textup{for all}\ C\in\mathcal C_1\tq |C|\geq 3,\ k\in C\\
x_{ij}+x_{ji}\leq 1&\textup{for all } ij\in A_1
\end{array}
\right\},$$

Therefore,
\begin{align*}
\Cl(P_\dl(V_\mtz))&=\bigcap_{k\in N_1}P_\dl(d^k)\\
&=\left\{x\in P_\ap\tq 
\begin{array}{rl}
\sum_{ij\in C}(x_{ij}+x_{ji})-x_{k k^C_-}-x_{k^C_+ k}\leq |C|-1&  \textup{for all}\ C\in\mathcal C_1,\ \textup{for all}\ k\in C\tq |C|\geq 3\\
x_{ij}+x_{ji}\leq 1&\textup{for all}\ ij\in A_1
\end{array}
\right\}.
\end{align*}
\end{proof}

Observe that the L1RMTZ formulation in \cite{gouveia1999asymmetric} coincides with $\Cl(P_\dl(V_\mtz))$. However, from Theorem~\ref{cl_dl} and Proposition~\ref{dl_mtz}, we have that L1RMTZ does not yield $\Cl(P_\dl(D))$ in general.


\section{Parametric $b$-SCF formulations}\label{sec:scf}
In this section, we consider flow-based formulations for the ATSP parametrized by demand and supply. We apply a framework similar to that of $d$-MTZ and $d$-DL formulations.

Given $b\in\RR^{N_1}_{++}$, let $M=\sum_{i\in N_1}b_i$ and consider the following generalization of the SCF formulation
\begin{eqnarray}
\sum_{ij\in\delta^+(i)}f_{ij}-\sum_{ji\in\delta^-(i)}f_{ji}=\begin{cases}M& i=1\\-b_i& i\in N_1\end{cases}\nonumber\\
f_{ij}\leq Mx_{ij}\quad \textup{for all}\ ij\in A.\nonumber
\end{eqnarray}
Note that the set $B=\{b\in\RR^{N_1}_{++}\tq \sum_{i\in N_1}b_i=1\}$ is obtained by the normalizing condition $M=1$. For $b\in \overline B$, recall 
$$Q_\scf(b)=\left\{(x,f)\in P_\ap\times\RR^A_+\tq \begin{array}{rl}\displaystyle\sum_{ij\in\delta^+(i)}f_{ij}-\sum_{ji\in\delta^-(i)}f_{ji}=-b_i& \textup{for all}\ i\in N_1\\
\displaystyle f_{ij}\leq x_{ij}& \textup{for all}\ ij\in A
\end{array}\right\},$$
where we have omitted the flow constraint on node $i=1$ since it is redundant.

\subsection{Projecting $Q_\scf(b)$}
We now give an explicit description of the set $P_\scf(b) = \proj_x(Q_\scf(b))$.

\begin{proposition}\label{prop:Pscf(b)}
We have that
$$P_\scf(b)=\left\{x\in P_{\text{AP}}\tq \sum_{ij\in\delta^+(S)}x_{ij}\geq \sum_{i\in S}b_i\   \textup{for all } S\in\mathcal S_1\right\}.$$
\end{proposition}
\begin{proof}

Let $x\in P_\scf(b)$ and let $\hat b=(-1,b)\in\RR^N$. Since $\sum_{i\in N}\hat b_i=0$, by Gale's flow theorem \cite{gale1957theorem}, there exists $f\in\RR^A_+$ such that $f\leq x$ and $\sum_{ij\in\delta^-(i)}f_{ij}-\sum_{ji\in\delta^+(i)}f_{ji}=\hat b_i$ for $i\in N$ if and only if $\sum_{ij\in\delta^-(S)}x_{ij}\geq \sum_{i\in S}\hat b_i$ for all $\emptyset\neq S\subsetneq N$. Note that if $1\in S$, then $\sum_{i\in S}\hat b_i=\sum_{i\in S\setminus\{1\}}b_i-1\leq 0$ and the condition is trivially satisfied. In addition, since $x\in P_\ap$, we have $\sum_{ji\in\delta^-(i)}x_{ji}=1\geq b_i$ for $i\in N_1$, and thus the condition is satisfied if $S=\{i\}$. Therefore, the existence of $f$ holds if and only if $\sum_{ij\in\delta^-(S)}x_{ij}\geq \sum_{i\in S}b_i$ for all $S\in\mathcal S_1$. Finally, since $x\in P_\ap$, we have $\sum_{ij\in\delta^-(S)}x_{ij}=\sum_{ij\in\delta^+(S)}x_{ij}$, which completes the proof.
\end{proof}

Noting that any $x\in P_\ap$ satisfies $\sum_{ij\in A(S)}x_{ij}+\sum_{ij\in\delta^+(S)}x_{ij}=|S|$ for all $S\in\mathcal S_1$, we have
\begin{equation}\label{eq:SCF_proj}
   P_\scf(b)=\left\{x\in P_{\text{AP}}\tq \sum_{ij\in A(S)}x_{ij}\leq |S|-\sum_{i\in S}b_i\   \textup{for all } S\in\mathcal S_1\right\}. 
\end{equation}
The inequalities in~\eqref{eq:SCF_proj} generalize the Weak clique inequalities \eqref{eq:weak_clique_ineq}.

\begin{proposition}
For each $b\in \overline B$ and $S\in\mathcal S_1$, the inequality $\sum_{ij\in\delta^+(S)}x_{ij}\geq\sum_{i\in S}b_i$ defines a facet of $P_\scf(b)$ if and only if $\sum_{i\in S}b_i>0$ and $S\neq N_1$.
\end{proposition}
\begin{proof}
Given $\hat S\in\mathcal S_1$ with $\sum_{i\in S}b_i>0$ and $\hat S\neq N_1$, let $\hat C$ and $\bar C$ be directed cycles covering all nodes in $\hat S$ and in $N\setminus \hat S$, respectively, and consider $\hat x=\ones_{\hat C}+\ones_{\bar C}\in\{0,1\}^A$. We have that $\hat x\in P_\ap$ and satisfies all the constraints in the definition of $P_\scf(b)$, with the exception of that associated to $\hat S$ as $\sum_{ij\in\delta^+(S)}\hat x_{ij}=0$, and therefore this constraint is non-redundant. In addition, this inequality is satisfied strictly by $\frac{1}{n-1}\ones \in P_\scf(b)$ since
$$\sum_{ij\in \delta^+(S)}\frac{1}{n-1}=|S|(n-|S|)\frac{1}{n-1}\geq 2(n-2)\frac{1}{n-1}>1\geq\sum_{i\in S}b_i,$$
where the first inequality follows from $2\leq|S|\leq n-2$ and the second from $n\geq 4$. We conclude the constraint defines a facet of $P_\scf(b)$. 

If $\sum_{i\in S}b_i=0$, the inequality is implied by $x_{ij}\geq 0$ for $ij\in\delta^+(S)$. If $S=N_1$, the corresponding inequality is redundant (see Observation \ref{obs_c1}).
\end{proof}

\subsection{Comparing $P_\scf(b)$ for different $b\in \bar B$}

The case of $b$-SCF formulations is somewhat simpler than the case of $d$-MTZ or $d$-DL formulations. The following result shows that $b$-SCF formulations are never comparable.

\begin{proposition}
For any $b,b'\in B$ with $b\neq b'$, $P_\scf(b)$ and $P_\scf(b')$ are not comparable.
\end{proposition}
\begin{proof}
Given $b,b'\in B$ with $b\neq b'$, let $\eta=b'-b$. Without loss of generality, assume that $\eta_2\leq \eta_3\leq\cdots\leq\eta_{n-1}\leq\eta_n$. Since $b\neq b'$ and $\sum_{i\in N_1}\eta_i=0$, we have $\eta_2<0$ and $\eta_n>0$. We will show now that $\eta_2+\eta_3<0$. If $\eta_3\leq0$, we are done. If $\eta_3>0$, then $\sum_{i\in N_1\setminus\{2,3\}}\eta_i>0$, and we obtain $\eta_{2}+\eta_{3}=-\sum_{i\in N_1\setminus\{2,3\}}\eta_i<0$, as desired. Similarly, we have $\eta_{n-1}+\eta_n>0$. Let $S=\{2,3\}$ and $S'=\{n-1,n\}$. We have $\sum_{ij\in\delta^+(S)}x_{ij}\geq\sum_{i\in S}b_i>\sum_{i\in S}b'_i$. Since $\sum_{ij\in\delta^+(S)}x_{ij}\geq\sum_{i\in S}b'_i$ and $\sum_{ij\in\delta^+(S)}x_{ij}\geq\sum_{i\in S}b_i$ define facets of $P_\scf(b')$ and $P_\scf(b)$, respectively, this shows that $P_\scf(b)\nsubseteq P_\scf(b')$. Analogously, we have $\sum_{ij\in\delta^+(S')}x_{ij}\geq\sum_{i\in S'}b'_i>\sum_{i\in S'}b_i$. Since $\sum_{ij\in\delta^+(S')}x_{ij}\geq\sum_{i\in S'}b_i$ and $\sum_{ij\in\delta^+(S')}x_{ij}\geq\sum_{i\in S'}b'_i$ define facets of $P_\scf(b)$ and $P_\scf(b')$, respectively, this shows that $P_\scf(b')\nsubseteq P_\scf(b)$.
\end{proof}

\subsection{The intersection of all $b$-SCF formulations}\label{sec:closure_d-SCFformulations}


Given nonempty $B'\subseteq \overline B$, let
$$\Ef(Q_\scf(B'))=\left\{(x,f)\in P_\ap\times\left(\RR^{A}\right)^{B'}\tq (x,f(b))\in Q_\scf(b)\ \forall b\in B'\right\}.$$

\begin{proposition}\label{prod_SCF}
    If for each $i\in N_1$, there exists $b\in B'$ such that $b_i>0$, then no integer $x\in \proj_x(\Ef(Q_\scf(B')))$ can define subtours. 
    In particular, if $B'$ is finite, then $\Ef(Q_\scf(B'))$ is a valid extended formulation for the ATSP.
\end{proposition}

\begin{proof}
    We have that $\proj_x(\Ef(Q_\scf(B')))=\bigcap_{b\in B'}\proj_x(Q_\scf(b))$. Therefore, by Proposition \ref{prop:Pscf(b)}, for any $x\in\proj_x(Q_\scf(B'))$ and $S\in\mathcal S_1$, it holds that $\sum_{ij\in\delta^+(S)}x_{ij}\geq \sum_{i\in S}b_i$ for all $b\in B'$, or equivalently, $\sum_{ij\in A(S)}x_{ij}\leq |S|-\max\{\sum_{i\in S}b_i\tq b\in B'\}$. By hypothesis, the right-hand side is strictly less than $|S|$. Thus, no integer $x\in \proj_x(Q_\scf(B'))$ can define subtours and the last assertion follows.
\end{proof}

For each $k\in N_1$, let $b^k\in\overline B$ be the $k$-canonical vector, and define $V_\scf=\{b^k\tq k\in N_1\}$. Clearly, $V_\scf$ is the set of vertices of $\overline B$ (see Observation~\ref{obs_nphard}).

\begin{theorem}\label{cl_scf}
We have that $\Cl(P_\scf(B))=\Cl(P_\scf(V_\scf))=\overline P_\scf$, where
$$\overline P_\scf=\left\{x\in P_{\text{AP}}\tq \sum_{ij\in\delta^+(S)}x_{ij}\geq 1\   \textup{for all } S\in\mathcal S_1\right\}.$$
\end{theorem}

\begin{proof}
We have $\overline P_\scf\subseteq P_\scf(b)$ for all $b\in B$, and thus $\overline P_\scf\subseteq \Cl(P_\scf(B))=\Cl(P_\scf(\overline B))$. To show that $\Cl(P_\scf(\overline B))=\overline P_\scf$, note that for any $k\in N_1$,
\begin{align*}
P_\scf(b^k)&=\left\{x\in P_{\text{AP}}\tq \sum_{ij\in\delta^+(S)}x_{ij}\geq \sum_{i\in S}b^k_i\   \textup{for all } S\in\mathcal S_1\right\}\\
&=\left\{x\in P_{\text{AP}}\tq \sum_{ij\in\delta^+(S)}x_{ij}\geq 1\   \textup{for all } S\in\mathcal S_1\tq k\in S\right\}.
\end{align*}
Therefore,
\begin{align*}
\Cl(P_\scf(V_\scf))&=\bigcap_{k\in N_1}P_\scf(b^k)\\
&=\left\{x\in P_{\text{AP}}\tq \sum_{ij\in\delta^+(S)}x_{ij}\geq 1\   \textup{for all } S\in\mathcal S_1,\ \textup{for all } k\in N_1\tq k\in S\right\}\\\
&=\left\{x\in P_{\text{AP}}\tq \sum_{ij\in\delta^+(S)}x_{ij}\geq 1\   \textup{for all } S\in\mathcal S_1\right\}=\overline P_\scf.
\end{align*}
\end{proof}
Note that $\overline P_\scf$ can be equivalently written as
$$\overline P_\scf=\left\{x\in P_{\text{AP}}\tq \sum_{ij\in A(S)}x_{ij}\leq |S|-1\ \textup{for all } S\in\mathcal S_1\right\},$$
and in particular, we recover the DFJ formulation. Below we recover  MCF as an extended formulation.

\begin{corollary}
We have that $\overline P_\scf=\proj_x(\overline Q_\scf)$, where
$$\overline Q_\scf=\left\{(x,(f^k\tq k\in N_1))\in P_\ap\times\RR^{N_1\times A}_+\tq \begin{array}{rl}\displaystyle\sum_{ij\in\delta^+(i)}f^k_{ij}-\sum_{ji\in\delta^-(i)}f^k_{ji}=-\delta_{ik}& \textup{for all}\ i,k\in N_1\\
\displaystyle f_{ij}\leq x_{ij}& \textup{for all}\ ij\in A
\end{array}\right\}.$$
and $\delta_{ik}=1$ if $i=k$ and $\delta_{ik}=0$ else.
\end{corollary}
\begin{proof}
Noting that $\overline Q_\scf=\Ef Q_\scf(V_\scf))$, the result follows from Theorem~\ref{cl_scf} and Corollary~\ref{prod_V}.
\end{proof}

\section{Comparing closures}\label{sec:comp_closures}

Below, we formalize how the closures we have introduced in previous sections relate.

\begin{proposition}\label{prop:comparing_closures}
The following hold:
\begin{enumerate}
    \item For $n\geq 4$, $\Cl(P_\scf(B))\subseteq\Cl(P_\dl(D))\subseteq\Cl(P_\dl(V_\mtz))\subseteq\Cl(P_\mtz(D))$.
    \item For $n=4$, $\Cl(P_\scf(B))=\Cl(P_\dl(D))=\Cl(P_\dl(V_\mtz))=\Cl(P_\mtz(D))$.
    \item For $n\geq 5$, for any $\hat C\in\mathcal C_1$ with $3\leq|\hat C|\leq n-2$ and for any $k\in \hat C$, there exists $\hat x\in \Cl(P_\mtz(D))$ that violates $\sum_{ij\in \hat C}(x_{ij}+x_{ji})-x_{k^+k}-x_{kk^-}\leq |\hat C|-1$. In particular, $\Cl(P_\dl(V_\mtz))\subsetneq \Cl(P_\mtz(D))$.
    \item For $n\geq 5$, for any $\hat C\in\mathcal C_1$ with $3\leq|\hat C|\leq n-2$ and for any $kl\in \hat C$, there exists $\hat x\in \Cl(P_\dl(V_\mtz))$ that violates $\sum_{ij\in \hat C}(x_{ij}+x_{ji})-x_{lk}\leq |\hat C|-1$. In particular, $\Cl(P_\dl(D))\subsetneq \Cl(P_\dl(V_\mtz))$.
    \item For $n\geq 5$, for any $\hat S\in\mathcal S_1$ with $3\leq|S|\leq n-2$, there exists $\hat x\in \Cl(P_\dl(D))$ that violates $\sum_{ij\in A(S)}x_{ij}\leq |\hat C|-1$. In particular, $\Cl(P_\scf(B))\subsetneq \Cl(P_\dl(D))$.
\end{enumerate}
\end{proposition}
\begin{proof}

\begin{enumerate}
    \item Given $S\in\mathcal S_1$ and $C\in\mathcal C_1$ with $C\subseteq A(S)$ and $|C|=|S|\geq 3$, we have $\sum_{ij\in A(S)}x_{ij}\geq \sum_{ij\in C}(x_{ij}+x_{ji})-x_{k^+k}\geq \sum_{ij\in C}(x_{ij}+x_{ji})-x_{k^+k}-x_{kk^-}\geq \sum_{ij\in C }x_{ij}$ for any $x\in P_\ap$ and $k\in S$.
    
    \item From Observation~\ref{obs_c1}, we have that all four above sets are equal to $\{x\in P_\ap\tq x_{ij}+x_{ji}\leq 1\ \textup{for all}\ ij\in A_1\}$ since the remaining constraints are redundant in each case.
    
    \item Given $\hat C\in\mathcal C_1$ with $3\leq |\hat C|\leq n-2$, let $\tilde C\subseteq A$ be any directed cycle on the remaining nodes of $N$, which are at least 2. Define $\hat x\in\RR^A$ by $\hat x_{ij}=\frac{|\hat C|-1}{|\hat C|}$ if $ij\in \hat C$, $\hat x_{ij}=\frac{1}{|\hat C|}$ if $ji\in \hat C$, $\hat x_{ij}=1$ if $ij\in\tilde C$, and $\hat x_{ij}=0$ else. Clearly, we have $\hat x\in P_\ap$. Now, let $C\in\mathcal C_1$. If $C=\hat C$, then $\sum_{ij\in C}\hat x_{ij}= |C|-1$. If $C\neq \hat C$, then $\sum_{ij\in C}\hat x_{ij}= \sum_{ij\in C\cap(\hat C\cup \tilde C)}\hat x_{ij}\leq |C|-1$. Therefore, $\hat x\in\Cl(P_\mtz(D))$. However, $\sum_{ij\in \hat C}(\hat x_{ij}+\hat x_{ji})-\hat x_{k^+k}-\hat x_{kk^-}=|\hat C|-\frac{2}{|\hat C|}>|\hat C|-1$ for any $k\in \hat C$.

    \item Given $\hat C\in\mathcal C_1$ with $3\leq |\hat C|\leq n-2$, let $\tilde C\subseteq A$ be any directed cycle on the remaining nodes of $N$, which are at least 2. Define $\hat x\in\RR^A$ by $\hat x_{ij}=\frac{1}{2}$ if $ij\in \hat C$ or $ji\in \hat C$, $\hat x_{ij}=1$ if $ij\in\tilde C$, and $\hat x_{ij}=0$ else. Clearly, we have $\hat x\in P_\ap$ and $\hat x_{ij}+\hat x_{ji}\leq 1$ for all $ij\in A_1$. Now, let $C\in\mathcal C_1$ with $|C|\geq 3$ and let $k\in C$. If $C=\hat C$, then $\sum_{ij\in C}(\hat x_{ij}+\hat x_{ji})-\hat x_{k k^C_-}-\hat x_{k^C_+ k}= |C|-1$. If $C\neq \hat C$, then $\sum_{ij\in C}(\hat x_{ij}+\hat x_{ji})-\hat x_{k k^C_-}-\hat x_{k^C_+ k}\leq \sum_{ij\in C\cap(\hat C\cup \tilde C)}\hat x_{ij}\leq |C|-1$. Therefore, $\hat x\in\Cl(P_\dl(V_\mtz))$. However, $\sum_{ij\in \hat C}(\hat x_{ij}+\hat x_{ji})-\hat x_{lk}=|\hat C|-\frac{1}{2}>|\hat C|-1$ for any $kl\in \hat C$.

    \item Given $\hat S\in\mathcal S_1$ with $3\leq |\hat S|\leq n-2$, let $\hat C\in \mathcal C_1$ be any directed cycle on $\hat S$. Also, let $\tilde C\subseteq A$ be any directed cycle on the remaining nodes of $N$, which are at least 2. Let $h\in N\setminus \hat S$ be the node succeeding node 1 in $\tilde C$ and define $\hat x\in\RR^A$ by $\hat x_{ij}=\frac{|\hat C|-1}{2|\hat C|-1}$ if $ij\in \hat C$ or $ji\in \hat C$, $\hat x_{1j}=\frac{1}{2|\hat C|-1}$ if $j\in \hat S$, $\hat x_{ih}=\frac{1}{2|\hat C|-1}$ if $i\in \hat S$, $x_{1h}=\frac{|\hat C|-1}{2|\hat C|-1}$, $\hat x_{ij}=1$ if $ij\in \tilde C\setminus\{1h\}$, and $\hat x_{ij}=0$ else. Clearly, we have $\hat x\in P_\ap$ and $\hat x_{ij}+\hat x_{ji}\leq 1$ for all $ij\in A_1$. Now, let $C\in\mathcal C_1$ with $|C|\geq 3$ and let $kl\in C$. If $C=\hat C$, then $\sum_{ij\in C}(\hat x_{ij}+\hat x_{ji})-\hat x_{lk}=2|C|\frac{|C|-1}{2|C|-1}-\frac{|C|-1}{2|C|-1}= |C|-1$. If $C\neq \hat C$, then $\sum_{ij\in C}(\hat x_{ij}+\hat x_{ji})-\hat x_{lk}\leq \sum_{ij\in C\cap [\hat C\cup \tilde C\cup\delta^+(1)\cup\delta^-(h)]}\hat x_{ij}\leq |C|-1$. Therefore, $\hat x\in\Cl(P_\dl(D))$. However, $\sum_{ij\in A(\hat S)}\hat x_{ij}=|\hat C|\frac{2|\hat C|-2}{2|\hat C|-1}=|\hat C|-\frac{|\hat C|}{2|\hat C|-1}>|\hat C|-1$.
\end{enumerate}
\end{proof}

For a graph with five or more nodes ($n\geq 5$), the results in Proposition~\ref{prop:comparing_closures} can be summarized in the following corollary. 

\begin{corollary}
    For $n\geq 5$, $\Cl(P_\scf(B))\subsetneq\Cl(P_\dl(D))\subsetneq\Cl(P_\dl(V_\mtz))\subsetneq\Cl(P_\mtz(D))$.
\end{corollary}

\section{Final remarks}
In this work, we have introduced parametric formulations for the ATSP based on the MZT, DL, and SCF formulations. We showed how these formulations arise and how they compare in terms of the choice of parameters and in terms of their closures.

A natural question is how to apply these results computationally to solve instances of the ATSP. We envision a scheme where we start with the extended formulation given by a single element of a parametric family (a particular choice of the parameters), and then we enrich this formulation by including variables and constraints associated to more elements of the family. It would be interesting to study how to select the parameters in a dynamic fashion, as we might need criteria different from the notion of a most-violated constraint given a fractional solution.

Another research direction is to apply the framework of parametric formulations to other combinatorial problems whose natural or classic formulations can be generalized.


\section*{Statements and Declarations}

\subsection*{Competing interests}
The authors declare that they have no competing interests.

\subsection*{Funding}
 We would  like to thank for the support from the  ANID grant  Fondecyt \# 1210348.

\bibliographystyle{plain}
\bibliography{mtz}

\end{document}